\pdfoutput = 1

\documentclass[english]{article}
\usepackage{lmodern}
\usepackage[T1]{fontenc}
\usepackage[latin9]{inputenc}
\usepackage{geometry}
\usepackage{babel}
\usepackage{float}
\usepackage{amsmath}
\usepackage{amsthm}
\usepackage{amssymb}
\usepackage{graphicx}
\usepackage[unicode=true,pdfusetitle,
 bookmarks=true,bookmarksnumbered=false,bookmarksopen=false,
 breaklinks=false,pdfborder={0 0 1},backref=false,colorlinks=false]
 {hyperref}

\makeatletter

\newcommand{\noun}[1]{\textsc{#1}}
\providecommand{\tabularnewline}{\\}
\floatstyle{ruled}
\newfloat{algorithm}{tbp}{loa}
\providecommand{\algorithmname}{Algorithm}
\floatname{algorithm}{\protect\algorithmname}

\numberwithin{equation}{section}
\theoremstyle{plain}
\newtheorem{thm}{\protect\theoremname}[section]
\theoremstyle{plain}
\newtheorem{assumption}[thm]{\protect\assumptionname}
\theoremstyle{definition}
\newtheorem{defn}[thm]{\protect\definitionname}
\theoremstyle{definition}
\newtheorem{example}[thm]{\protect\examplename}
\theoremstyle{plain}
\newtheorem{prop}[thm]{\protect\propositionname}
\theoremstyle{plain}
\newtheorem{lem}[thm]{\protect\lemmaname}
\theoremstyle{remark}
\newtheorem{rem}[thm]{\protect\remarkname}
\theoremstyle{plain}
\newtheorem{fact}[thm]{\protect\factname}

\usepackage{cite}
\usepackage{amsmath,amssymb,amsfonts}
\usepackage{algorithmic}
\usepackage{graphicx}
\usepackage{textcomp}
\def\BibTeX{{\rm B\kern-.05em{\sc i\kern-.025em b}\kern-.08em
    T\kern-.1667em\lower.7ex\hbox{E}\kern-.125emX}}

\makeatother

\providecommand{\assumptionname}{Assumption}
\providecommand{\definitionname}{Definition}
\providecommand{\examplename}{Example}
\providecommand{\factname}{Fact}
\providecommand{\lemmaname}{Lemma}
\providecommand{\propositionname}{Proposition}
\providecommand{\remarkname}{Remark}
\providecommand{\theoremname}{Theorem}

\begin{document}
\title{Stochastic approximation for risk-aware Markov decision processes}
\author{Wenjie Huang \thanks{Wenjie Huang (wenjiehuang@cuhk.edu.cn) is an International Postdoctoral Fellow at Shenzhen Research Institute of Big Data (SRIBD) and Institute for Data and Decison Analytics, The Chinese University of Hong Kong, Shenzhen. His research was supported by SRIBD International Postdoctoral Fellowship, and the National Research Foundation (NRF), Prime Ministers Office, Singapore under its Campus for Research Excellence and Technological Enterprise (CREATE) program.} \and William B. Haskell \thanks{William B. Haskell (whaskell@purdue.edu) is an Assistant Professor in Krannert School of Management at Purdue University. His research was supported by Singapore Ministry of Education Grant R-266-000-083-133 and by Singapore Ministry of Education Tier II Grant MOE2015-T2-2-148.}}
\maketitle
\begin{abstract}
We develop a stochastic approximation-type algorithm
to solve finite state/action, infinite-horizon, risk-aware Markov
decision processes. Our algorithm has two loops. The inner loop computes
the risk by solving a stochastic saddle-point problem. The outer loop performs
$Q$-learning to compute an optimal risk-aware policy. Several widely investigated risk measures (e.g. conditional value-at-risk,
optimized certainty equivalent, and absolute semi-deviation) are covered by our algorithm. Almost sure convergence and the convergence rate of the algorithm are established.
For an error tolerance $\epsilon>0$ for the optimal $Q$-value estimation gap and learning rate $k\in(1/2,\,1]$, the overall
convergence rate of our algorithm is $\Omega((\ln(1/\delta\epsilon)/\epsilon^{2})^{1/k}+(\ln(1/\epsilon))^{1/(1-k)})$
with probability at least $1-\delta$\emph{.}\\
\emph{}\\
\emph{Keywords:} Markov decision processes; Risk measure; Saddle-point;
Stochastic approximation; $Q$-learning.
\end{abstract}

\section{Introduction }

The analysis of complex systems such as inventory control, financial
markets, waste-to-energy plants, and computer networks is difficult
because of the intrinsic uncertainty in these systems. Risk-aware
optimization offers a possible remedy by searching for strong reliability
guarantees. In particular, it gives more attention to low probability
but high cost events than a risk-neutral optimizer would. Risk awareness
is especially important in sequential decision-making.

Markov decision processes (MDPs) introduced by Bellman in \cite{Bellman1957}
provide a mathematical framework for sequential decision-making. However,
the exact model of the underlying MDP is often unknown and one can
only observe the trajectory of states, actions, and costs. $Q$-learning,
as developed in \cite{Watkins1992}, can produce an optimal policy
in a model-free way based only on observed trajectories.

In this paper, we synthesize the work on risk-aware optimization with
reinforcement learning, specifically, $Q$-learning. As our main contribution,
we develop a novel asynchronous stochastic-approximation type algorithm
to solve infinite-horizon risk-aware MDPs. This algorithm can compute
the risk-aware optimal policy based only on observations, without
any knowledge of the explicit form of the cost function or the transition
probabilities. 

\subsection{Literature review}

\subsubsection{Risk measures}

In general, a risk measure is a mapping from random variables to scalars.
It can be interpreted as the amount of an asset (traditionally currency)
to be kept in reserve to make the risk acceptable. The literature
emphasizes convex and coherent risk measures. In \cite{Ruszczynski2006},
a theory of convex analysis is developed for optimization of convex
risk measures. Several specific examples of convex and coherent risk
measures along with their various risk envelopes are given in \cite[Chapter 6]{Shapiro2009}.

Special attention has been given to the class of coherent and law
invariant risk measures, the most well known being conditional value-at-risk
(CVaR). Optimization of CVaR is studied in \cite{Rockafellar2000},
which reveals that CVaR has many desirable properties for stochastic
optimization. The most famous representation result for law-invariant
coherent risk measures is the Kusuoka representation (see \cite{Shapiro2013}
for example) which shows that such risk measures can be 'built' out
of CVaR. There are several other important classes of risk measures
such as: optimized certainty equivalent \cite{Ben-Tal2007}, spectral
risk measures \cite{Acerbi2002}, distortion risk measures \cite{Bertsimas2009},
and entropic risk measures \cite{Foellmer2011}.

Numerical methods for risk-aware optimization are critical for practical
application. In \cite{Krokhmal2002}, CVaR-constrained optimization
problems are solved with a combination of discretization, linearization,
and sample average approximation. For multistage CVaR optimization,
\cite{Philpott2013} uses the dual representation of general coherent
risk measures to develop sampling-based algorithms. In \cite{Bardou2009},
stochastic approximation is used to estimate CVaR in data-driven optimization.
In addition, in \cite{Carbonetto2009}, stochastic interior-point
algorithms are developed for risk-constrained optimization.

\subsubsection{Risk-aware MDPs}

Risk-aware MDPs have been widely studied. In \cite{Baeuerle2011},
the authors minimize the conditional value-at-risk of the discounted
cost over both the finite and infinite horizon. In the follow-up work
\cite{Baeuerle2013}, the authors minimize a certainty equivalent
of the total cost for both finite and infinite horizon problems. Dynamic
programming methods are developed in both \cite{Baeuerle2011} and
\cite{Baeuerle2013}. A CVaR-constrained MDP is solved with both offline
and online iterative algorithms in \cite{Borkar2014}. In \cite{Chow2018},
both risk and modeling errors are taken into account in an MDP framework
for risk-sensitive and robust decision-making, and an approximate
value-iteration type algorithm is presented.

In \cite{Haskell_Markov_2012}, the authors study stochastic dominance-constrained
MDPs, and show that this class of MDPs can be reformulated as linear
programming problems using the convex analytic approach. In \cite{Haskell2015a},
the authors develop the convex analytic approach for a general family
of risk-aware MDPs. 

Dynamic programming equations are developed for a wide class of risk-aware
MDPs in \cite{Ruszczynski2010}, and corresponding value iteration
and policy iteration algorithms are developed. The family of risk
measures studied in this work are often called ``dynamic risk measures''
or ``Markov risk measures'', and are notable for satisfying the
property of time-consistency. In \cite{Shen2013}, the theory of risk-sensitive
MDPs is developed based on iterative risk measures which only depend
on the current state, rather than on the whole history. 

In \cite{Chow2018}, reinforcement learning algorithms for percentile
risk-constrained MDPs are proposed. In \cite{Prashanth2014,Tamar2014},
policy gradient algorithms are applied to MDPs with CVaR appearing
in either the objective or constraints. In \cite{Jiang2017}, a specific
class of risk measures called quantile-based risk measures is proposed
for MDPs and a simulation-based approximate dynamic programming (ADP)
algorithm is developed for the resulting problem. This paper emphasizes
importance sampling, to direct samples toward the \emph{risky region}
as the ADP algorithm progresses. In \cite{Shen2014}, a risk-sensitive
reinforcement learning algorithm based on utility functions is investigated.
A similar technique is applied to the risk-sensitive control of finite
MDPs in \cite{Borkar2002}. 

\subsubsection{Stochastic approximation and $Q$-learning}

$Q$-learning is introduced in \cite{Watkins1992}. The idea of $Q$-learning
is to use the observed transitions and costs to compute the optimal
policy (so that exact knowledge of the underlying MDP model is not
needed). In \cite{Tsitsiklis1994}, a thorough convergence proof of
the $Q$-learning algorithm is given based on stochastic approximation
and the theory of parallel asynchronous algorithms (see \cite{Borkar2008}
for more details on the theory of stochastic approximation). $Q$-learning
has wide applications in the areas of robotics and operations management,
and has also recently been applied to stochastic games
\cite{Hu2003}. Stochastic approximation has also been applied to
solve static stochastic optimization problems. In \cite{Nemirovski2005,Nemirovski2009},
efficient and robust stochastic approximation algorithms are developed
to solve saddle-point problem and optimize non-smooth functions.

\subsection{Contributions}

As our main contribution, we develop a stochastic approximation-type
algorithm for infinite-horizon risk-aware MDPs that covers a wide
range of risk measures. This algorithm is model-free and
it can compute the risk-aware optimal policy based only on observations.
We make the following three specific contributions:
\begin{enumerate}
\item \textbf{Generality} \textbf{of risk measure}s: There exists literature
(e.g. \cite{Borkar2002,Shen2014}) studying reinforcement learning
for risk-sensitive MDPs. The ``risk-sensitive'' objective in \cite{Borkar2002}
specifically refers to the expectation of the exponential function
of cumulative costs. In \cite{Shen2014}, the ``risk-sensitive''
objective is essentially utility-based shortfall, and in \cite{Baeuerle2013},
``risk-sensitive'' refers to utility-based certainty equivalent.
To the best of our knowledge, our present paper adds to the literature
by incorporating saddle-point risk measures. In \cite{Jiang2017}, only
quantile-based risk measures are included. In \cite{Baeuerle2011,Tamar2014},
the algorithms are specific to CVaR.
\item \textbf{Model-free asynchronous algorithm}: There exist several dynamic
programming based algorithms for solving risk-aware MDPs (see\cite{Baeuerle2011,Ruszczynski2010,Shen2013,Jiang2017,Tamar2014}),
but they all rely on some information about the underlying transitions
or cost function. Our novel stochastic approximation algorithm is
completely model-free. Our algorithm is also asynchronous, which means
that the $Q$-value is only updated when the corresponding state-action
pair is explored. This algorithm works even when no prior information
on the underlying MDP is available.
\item \textbf{Explicit sample complexity results}: We give a detailed convergence
rate analysis of our algorithm for both polynomial and linear learning
rates. We also show numerically that the convergence rate of our algorithm
is close to that of standard $Q$-learning. In \cite{Borkar2002,Jiang2017},
the almost sure convergence of the proposed algorithms is demonstrated,
but the explicit convergence rates are not derived.
\end{enumerate}
This paper is organized as follows. Section 2 reviews preliminaries
on risk measures and risk-aware MDPs. Section 3 then introduces saddle-point
risk measures and shows by example that many widely investigated risk
measures fall within this framework. Section 4 presents the details
of our algorithm as well as its almost sure convergence and convergence
rate. Section 5 contains the proofs of all our main theorems. We report
numerical experiments in Section 6 and then conclude the paper in
Section 7.

\section{Preliminaries}

This section introduces preliminary concepts and notations (listed
in Table 1). 
\begin{center}
\begin{table}
\caption{List of Key Notation }

\centering{}%
\begin{tabular}{cc}
\hline 
Notations & Definitions\tabularnewline
\hline 
$N,\,n$ & Outer iterations\tabularnewline
$T,\,t$ & Inner iterations\tabularnewline
$G$ & Objective function for saddle-point risk measure\tabularnewline
$\text{\ensuremath{\mathbb{S}}},\,\mathbb{A}$ & State and action space \tabularnewline
$\mathcal{T},\,\mathcal{T}_{G}$ & Risk-aware Bellman operator \tabularnewline
$\mathcal{R}_{(s,a)}^{G}$ & Risk measure with respect to state-action pair $(s,\,a)$ and function $G$\tabularnewline
$\mathcal{Y},\,\mathcal{Z}$ & Compact sets\tabularnewline
$K_{\mathcal{Y}},\,K_{\mathcal{Z}}$ & Euclidean diameters of $\mathcal{Y}$ and $\mathcal{Z}$.\tabularnewline
$L$ & Bounds for the subgradients\tabularnewline
$K_{G}$ & Constant of Lipschitz continuity for function $G$\tabularnewline
$K_{S}$ & Stability modulus of saddle-point\tabularnewline
$\mathcal{G}_{t}^{n}$ & The history of RaQL for $t\leq T$ and $n\leq N$\tabularnewline
$\epsilon_{t}^{n}$ & Risk estimation error, for $t\leq T$ and $n\leq N$\tabularnewline
$\xi_{t}^{n}$ & Approximation error, for $t\leq T$ and $n\leq N$\tabularnewline
$\tau$ & The iteration w.r.t sequence $D$\tabularnewline
$\tau_{m}$ & The iteration when the approximation error of $Q$-value is bounded
by $D_{m}$\tabularnewline
$D_{m}$ & A constructed sequence $D$ with time horizon $m$\tabularnewline
$Z_{t}^{n+1,\tau}$, $Y_{t}^{n+1,\tau}$ & Two random processes decomposed from $\left\{ Q_{t}^{n}\right\} $\tabularnewline
$\beta_{T}$ & Discount factor of sequence $D_{m}$\tabularnewline
$e$ & Natural logarithm\tabularnewline
$\|\cdot\|_{2}$ & $L_{2}$-norm\tabularnewline
$\|\cdot\|_{\infty}$ & Infinite norm\tabularnewline
\hline 
\end{tabular}
\end{table}
\par\end{center}

\subsection{Risk measures}

We begin with a probability space $(\Omega,\mathcal{F},\,P)$, where
$\Omega$ is a sample space, $\mathcal{F}$ is a $\sigma-$algebra
on $\Omega$, and $P$ is a probability measure on $(\Omega,\mathcal{\,F})$.
We work in $\mathcal{L}=L_{\infty}(\Omega,\mathcal{\,F},\,P)$, the
space of essentially bounded $\mathcal{F}-$measurable mappings.
For $X,\,Y\in\mathcal{L},$ $Y\succeq X$ means that $Y\text{(\ensuremath{\omega})\ensuremath{\geq}\ensuremath{X(\ensuremath{\omega})}}$
for all $\omega\in\Omega$. 

We define a risk measure\emph{ }to be a function $\rho:\,\mathcal{L}\rightarrow\mathbb{R}$,
which assigns to a random variable $X\in\mathcal{L}$ a real scalar
value $\rho(X)$. The following are four key properties of risk measures:

(A1) Monotonicity: If $X\succeq Y$, then $\rho(X)\geq\rho(Y)$.

(A2) Translation Invariance: $\rho(X+r)=\rho(X)+r$ for $r\in\mathbb{R}$.

(A3) Convexity: $\rho(\lambda X+(1-\lambda)Y)\leq\lambda\rho(X)+(1-\lambda)\rho(Y)$
for $0\leq\lambda\leq1.$

(A4) Positive Homogeneity: $\rho(\alpha X)=\alpha\rho(X)$ for $\alpha\geq0$.\\
These conditions were introduced in the pioneering paper \cite{Artzner1999}
and have since been heavily justified in other work including \cite{Ruszczynski2006,Natarajan2009,Bertsimas2009}.
Property (A1) states that a random variable with greater cost almost
surely must have higher risk. (A2) states that the addition of a certain
cost increases the risk by the same amount. (A3) gives precise meaning
to the idea that diversification should not increase risk. (A4) states
that the risk of a position is proportional to its size (i.e., if
we double our cost then we double our risk). A risk measure satisfying properties (A1)-(A3) is called a convex
risk measure, and a risk measure satisfying properties (A1)-(A4) is
called a coherent risk measure.

\subsection{Risk-aware MDPs}

A MDP is given by the tuple $\left(\mathbb{S},\,\mathbb{A},\,P,\,c\right)$
where $\mathbb{S}$ and $\mathbb{A}$ are the state and action spaces
and $\mathbb{K}:=\left\{ (s,\,a)\in\mathbb{S}\times\mathbb{A}\right\} $
is the set of all state-action pairs. Let $\mathcal{P}(\mathbb{S})$
be the space of probability measures over $\mathbb{S}$, and define
$\mathcal{P}(\mathbb{A})$ similarly. The transition law $P$ governs
the system evolution where $P(\cdot|s,\,a)\in\mathcal{P}(\mathbb{S})$
for all $(s,\,a)\in\mathbb{K}$, i.e., $P(s^{\prime}|s,\,a)$ for
$s^{\prime}\in\mathbb{S}$ is the probability of next visiting state
$s^{\prime}$ given the current state-action pair $(s,\,a)$. The
cost function $c:\,\mathcal{\mathbb{K}}\rightarrow\mathbb{R}$ gives
the cost of each state-action pair. Finally, $\gamma\in\left(0,\,1\right)$
is the discount factor. Let $\phi:\,\mathbb{S}\rightarrow\mathcal{P}(\mathbb{A})$ be a randomized policy.
For a given $\phi$, we obtain a stochastic process $\{(s_{t},\,a_{t})\}_{t\geq0}$
where $s_{t}$ and $a_{t}$ are the state and action at stage $t$,
respectively.

We make the following assumptions.
\begin{assumption}
\label{Assumption 2.1} (i) $\mathbb{S}$ and $\mathbb{A}$ are finite.

(ii) $0\leq c(s,\,a)\leq C_{\max}$ for all $(s,\,a)\in\mathbb{K}$.
Set $V_{\max}:=C_{\max}/(1-\gamma)$. 
\end{assumption}

Many real life MDPs satisfy Assumption \ref{Assumption 2.1}(i), including
machine replacement and sequential online auctions \cite{Hu2007},
critical infrastructure protection \cite{Ott2010}, wireless sensor
networks \cite{Alsheikh2015}, and human-robot interaction systems
\cite{Krsmanovic2006,Keizer2013}. 

In \cite{Ruszczynski2010}, the modern theory of risk measures is
adapted to MDPs. This class of risk-aware MDPs is constructed in the
following way. Denote our sequence of costs as $X_{t}=c(s_{t},\,a_{t})$
for all $t\geq0$. We begin by formalizing some details about the
risk of finite cost sequences $X_{t,\,T}:=\left(X_{t},\,X_{t+1},\ldots,\,X_{T}\right)$
before we consider the risk of the infinite cost sequence $X_{0},\,X_{1},\ldots$
actually faced by the controller. Let $\text{\ensuremath{\mathcal{L}_{t}:=\mathcal{L}_{\infty}(\Omega,\mathcal{\,F}_{t},\,P)}}$
and $\mathcal{L}_{t,\,T}:=\mathcal{L}_{t}\times\mathcal{L}_{t+1}\times\cdots\times\mathcal{L}_{T}$
for all $0\leq t\leq T<\infty$. 
\begin{defn}
\label{Definition 2.1-2}\cite[Definition 1]{Ruszczynski2010} For fixed $T\geq 1$ and $0\leq t\leq T$, (i)
A mapping $\rho_{t,\,T}:\,L_{t,\,T}\rightarrow L_{t}$, is called
a \emph{conditional risk measure} if: $\rho_{t,\,T}(Z_{t,\,T})\leq\rho_{t,\,T}(X_{t,\,T})$
for all $Z_{t,\,T},\,X_{t,\,T}\in\mathcal{L}_{t,\,T}$ such that $Z_{t,\,T}\leq X_{t,\,T}$. 

(ii) A \emph{dynamic risk measure} is a sequence of conditional risk
measures $\{\rho_{t,\,T}\}_{t=0}^{T}$.
\end{defn}

We now make our key assumptions about dynamic risk measures.
\begin{assumption}
\label{Assumption 2.5} For fixed $T\geq1$ and $0\leq t\leq T$, suppose the dynamic risk measure $\left\{ \rho_{t,\,T}\right\} _{t=0}^{T}$
satisfies the following conditions:

(i) (Normalization) $\rho_{t,\,T}(0,\,0,...,0)=0.$

(ii) (Conditional translation invariance) For any $X_{t,\,T}\in\mathcal{L}_{t,\,T}$,
\[
\rho_{t,\,T}(X_{t},\,X_{t+1},...,X_{T})=X_{t}+\rho_{t,\,T}(0,\,X_{t+1},...,X_{T}).
\]

(iii) (Convexity) For any $X_{t,\,T},\,Y_{t,\,T}\in\mathcal{L}_{t,\,T}$
and $0\leq\lambda\leq1$, $\rho_{t,\,T}(\lambda\,X_{t,\,T}+(1-\lambda)Y_{t,\,T})\leq\lambda\,\rho_{t,\,T}(X_{t,\,T})+(1-\lambda)\rho_{t,\,T}(Y_{t,\,T})$.

(iv) (Positive homogeneity) For any $X_{t,\,T}\in\mathcal{L}_{t,\,T}$
and $\alpha\geq0$, $\rho_{t,\,T}(\alpha\,X_{t,\,T})=\alpha\,\rho_{t,\,T}(X_{t,\,T}).$

(v) (Time-consistency) For any $X_{t,\,T},\,Y_{t,\,T}\in\mathcal{L}_{t,\,T}$
and $0\leq\tau\leq\theta\leq T$, the conditions $X_{k}=Y_{k}$ for
$k=\tau,...,\theta-1$ and $\rho_{\theta,\,T}(X_{\theta},....,X_{T})\leq\rho_{\theta,\,T}(Y_{\theta},...,Y_{T})$
imply $\rho_{\tau,\,T}(X_{\tau},...,X_{T})\leq\rho_{\tau,\,T}(Y_{\tau},...,Y_{T})$.
\end{assumption}

Many of these properties (monotonicity, convexity, positive homogeneity,
and translation invariance) were originally introduced for static
risk measures as properties (A1)-(A4). The next theorem gives a recursive
formulation for dynamic risk measures satisfying Assumption \ref{Assumption 2.5}.
This representation is the foundation of \cite{Ruszczynski2010} and
subsequent work on time-consistent dynamic risk measures. To express this result,
we define a mapping $\rho_{t}\text{ : }\mathcal{L}_{t+1}\rightarrow\mathcal{L}_{t}$
for $t\geq0$ to be a one-step (conditional) risk measure if $\rho_{t}(X_{t+1})=\rho_{t,\,t+1}(0,\,X_{t+1})$. 
\begin{thm}
\label{Theorem 2.3}\cite[Theorem 1]{Ruszczynski2010} Suppose Assumption
\ref{Assumption 2.5} holds, then
\begin{equation}
\rho_{t,\,T}(X_{t},\,X_{t+1},...,\,X_{T})=X_{t}+\rho_{t}(X_{t+1}+\rho_{t+1}(X_{t+2}+\cdot\cdot\cdot+\rho_{T}(X_{T}))),\label{Recursive}
\end{equation}
for all $0\leq t\leq T$, where $\rho_{t},\ldots,\,\rho_{T}$ are
one-step risk measures.
\end{thm}

Now we consider the risk of an infinite cost sequence. Following \cite{Ruszczynski2010},
the \emph{discounted} measure of risk $\rho_{t,\,T}^{\gamma}\text{ : }\mathcal{L}_{t,\,T}\rightarrow\mathbb{R}$
is defined via
\[
\rho_{t,\,T}^{\gamma}\left(X_{t},\,X_{t+1},\ldots,\,X_{T}\right):=\rho_{t,\,T}\left(\gamma^{t}X_{t},\,\gamma^{t+1}X_{t+1},\ldots,\,\gamma^{T}X_{T}\right).
\]
Define $\mathcal{L}_{t,\,\infty}:=\mathcal{L}_{t}\times\mathcal{L}_{t+1}\times\cdot\cdot\cdot$
for $t\geq0$ and $\rho^{\gamma}\text{ : }\mathcal{L}_{0,\,\infty}\rightarrow\mathbb{R}$
via
\[
\rho^{\gamma}\left(X_{0},\,X_{1},\ldots\right):=\lim_{T\rightarrow\infty}\rho_{0,\,T}^{\gamma}\left(X_{0},\,X_{1},\ldots\right).
\]
To provide our final representation result, we introduce the additional
assumption that risk preferences are stationary (they only depend
on the sequence of costs ahead, and are independent of the current
time).
\begin{assumption}
\label{assu:stationary_preferences} (Stationary preferences) For
all $T\geq1$ and $s\geq0$,
\[
\rho_{0,\,T}^{\gamma}\left(X_{0},\,X_{1},\ldots,\,X_{T}\right)=\rho_{s,\,T+s}^{\gamma}\left(X_{0},\,X_{1},\ldots,\,X_{T}\right).
\]
\end{assumption}

When Assumptions \ref{Assumption 2.5} and \ref{assu:stationary_preferences}
are satisfied, the corresponding dynamic risk measure is given by
the recursion:
\begin{align}
\rho^{\gamma}(X_{0},\,X_{1},...,\,X_{T},\ldots)=X_{0}+\rho_{1}(\gamma X_{1}+\rho_{2}(\gamma^{2}X_{2}+\cdot\cdot\cdot+\rho_{T}(\gamma^{T}X_{T})+\cdot\cdot\cdot)),\label{Recursive_infinite}
\end{align}
where $\rho_{1},\,\rho_{2},\ldots$ are all one-step risk measures.
Based on representation (\ref{Recursive_infinite}), we may evaluate
the risk of a policy $\phi$ via
\begin{equation}
J(\phi,\,s_{0}):=\rho\left(c(s_{0},\,a_{0})+\gamma\cdot\rho\left(c(s_{1},\,a_{1})+\gamma\cdot\rho\left(c(s_{2},\,a_{2})+\cdot\cdot\cdot\right)\right)\right),\label{Risk MDP}
\end{equation}
where $s_{0}$ is the initial state. To clarify, the same one-step
risk measure $\rho$ appears at all times $t\geq0$ due to the property of stationarity.
Formulation (\ref{Risk MDP}) explicitly captures the risk with respect
to the cost associated with the current state-action pair, as well
as the future risk. Let ${\Pi}$ denote the class
of deterministic stationary policies $\pi$ which map
from states to actions, i.e., $\pi:\,\mathcal{\mathbb{S}}\rightarrow\mathcal{\mathbb{A}}$. From \cite[Theorem 4]{Ruszczynski2010}, it shows that there exists an optimal deterministic stationary policy that minimizes Eq. (\ref{Risk MDP}). The corresponding risk-aware MDP is
\begin{equation}
\min_{\pi\in\Pi}\,J(\pi,\,s_{0}).\label{Risk MDP 2}
\end{equation}

\section{Saddle-point risk measures}

This section introduces the saddle-point representation
of risk measures. We elaborate on two main reasons for choosing this
representation. First, many widely investigated risk measures can
be represented as stochastic saddle-point problems including: conditional
value-at-risk, optimized certainty equivalent, absolute semi-deviation,
and functionally coherent risk measures. Second, there are efficient
algorithms for solving stochastic saddle-point problems (see \cite{Nemirovski2005,Nemirovski2009})
and thus for computing the risk.

To proceed, we now assume that $\Omega$ is Borel measurable and $\mathcal{L}$
is the set of all $X$ with bounded support $[\eta_{\min},\,\eta_{\max}]$ and $\eta_{\min},\,\eta_{\max}$ satisfying $-\infty < \eta_{\min} < \eta_{\max} < \infty$ (i.e., $X(\omega)\in[\eta_{\min},\,\eta_{\max}]$ for all $\omega\in\Omega$).
Take $\mathcal{Y}\subset\mathbb{R}^{d_{1}}$ and $\mathcal{Z}\subset\mathbb{R}^{d_{2}}$
to be closed and convex sets and define $K_{\mathcal{Y}},\,K_{\mathcal{Z}}$
to be the Euclidean diameters of $\mathcal{Y}$ and $\mathcal{Z}$,
respectively. For a proper function $G:\,\mathcal{L}\times\mathcal{Y}\times\mathcal{Z}\rightarrow\mathbb{R}$,
we consider the risk measure:
\begin{equation}
\rho(X)=\max_{z\in\mathcal{Z}}\min_{y\in\mathcal{Y}}\mathbb{E}_{P}\left[G(X,\,y,\,z)\right].\label{Saddle-2}
\end{equation}
We define $\partial_{y}G(\cdot,\,y,\,z)$ and $\partial_{z}G(\cdot,\,y,\,z)$
to be the subdifferentials of $G$ for all $(y,\,z)\in\mathcal{Y}\times\mathcal{Z}$,
and we define $G_{y}(\cdot,\,y,\,z)\in\partial_{y}G(\cdot,\,y,\,z)$
and $G_{z}(\cdot,\,y,\,z)\in\partial_{z}G(\cdot,\,y,\,z)$ to be particular
subgradients with respect to $y$ and $z$. We make the following
assumptions on the function $G$.
\begin{assumption}
\label{Assumption 3.3}\cite[Assumption B]{Nemirovski2005} (i) $\omega\rightarrow G(X(\omega),\,y,\,z)$
is $P$-square summable for every $y\in\mathcal{Y}$ and $z\in\text{\ensuremath{\mathcal{Z}}}$,
i.e., $\int_{\Omega}\left|G(X(\omega),\,y,\,z)\right|^{2}P(d\omega)<\infty,$
for all $(y,\,z)\in\mathcal{Y}\times\mathcal{Z}$.

(ii) $G$ is Lipschitz continuous on $\mathcal{L}\times\mathcal{Y}\times\mathcal{Z}$
with constant $K_{G}>1$.

(iii) $y\rightarrow G(X,\,y,\,z)$ is convex and $z\rightarrow G(X,\,y,\,z)$
is concave for all $(X,\,y,\,z)\in\mathcal{L}\times\mathcal{Y}\times\mathcal{Z}$.

(iv) Any selection of subgradients $\omega\rightarrow G_{y}(X(\omega),\,y,\,z)$
and $\omega\rightarrow G_{z}(X(\omega),\,y,\,z)$ is Borel measurable.
The subgradients $G_{y}(X,\,y,\,z)$ and $G_{z}(X,\,y,\,z)$ are uniformly
bounded, i.e., there exists $L>0$ such that $\|G_{y}(X,\,y,\,z)\|_{2}\leq L$
and $\|G_{z}(X,\,y,\,z)\|_{2}\leq L$ for all $(X,\,y,\,z)\in\mathcal{L}\times\mathcal{Y}\times\mathcal{Z}$. 
\end{assumption}

Under the assumption that $G$ is proper on $\mathcal{L}\times\mathcal{Y}\times\mathcal{Z}$
and Assumption \ref{Assumption 3.3}(iii), we know that the subdifferentials
$\partial_{y}G(\cdot,\,y,\,z)$ and $\partial_{z}G(\cdot,\,y,\,z)$
are non-empty for all $(y,\,z)\in\mathcal{Y}\times\mathcal{Z}$ by
\cite[Theorem 23.4]{Rockafellar1970}. Based on \cite[Theorem 7.47]{Shapiro2009}
and \cite[Remark 18]{Shapiro2009}, under Assumption \ref{Assumption 3.3}(i),
the subdifferentials $\partial_{y}\mathbb{E}_{P}\left[G(X,\,y,\,z)\right]$
and $\partial_{z}\mathbb{E}_{P}\left[G(X,\,y,\,z)\right]$ are nonempty
and satisfy $\partial_{y}\mathbb{E}_{P}\left[G(X,\,y,\,z)\right]=\mathbb{E}_{P}\left[\partial_{y}G(\cdot,\,y,\,z)\right]$, and
$\partial_{z}\mathbb{E}_{P}\left[G(X,\,y,\,z)\right]=\mathbb{E}_{P}\left[\partial_{z}G(\cdot,\,y,\,z)\right]$,
for all $x$ and $y$ in the interior of $\mathcal{Y}$ and $\mathcal{Z}$,
respectively. Thus, the subgradients $G_{y}(X,\,y,\,z)$ and $G_{z}(X,\,y,\,z)$
satisfy $\mathbb{E}_{P}G_{y}(X,\,y,\,z)\in\partial_{y}\mathbb{E}_{P}\left[G(X,\,y,\,z)\right]$
and $\mathbb{E}_{P}G_{z}(X,\,y,\,z)\in\partial_{z}\mathbb{E}_{P}\left[G(X,\,y,\,z)\right]$.
Under Assumptions \ref{Assumption 3.3}(i) and (iv), we see that
$\|\mathbb{E}_{P}G_{y}(X,\,y,\,z)\|_{2}$ and $\|\mathbb{E}_{P}G_{z}(X,\,y,\,z)\|_{2}$
are both bounded by $L$.

The following Theorem \ref{Theorem 3.2-1} provides sufficient conditions
for the saddle-point risk measure (\ref{Saddle-2}) to be a convex
risk measure satisfying axioms (A1)-(A3). In particular, we can find
a special class of functions $\{h_{z}\}_{z\in\mathcal{Z}}$ and then
construct $G$ from these $\{h_{z}\}_{z\in\mathcal{Z}}$. The proof
of Theorem \ref{Theorem 3.2-1} may be found in the Appendix.
\begin{thm}
\label{Theorem 3.2-1} Set $\mathcal{Y}=[\eta_{\min},\,\eta_{\max}]$
and let $\{h_{z}\}_{z\in\mathcal{Z}}$ be a collection of functions
such that $h_{z}(X,\,y)$ (for $X\in\mathcal{L}$ and $y\in\mathcal{Y}$),
parameterized by $z\in\mathcal{Z}$ that satisfies:

(i) $\omega\rightarrow h_{z}(X(\omega),\,y)$ is $P$-square summable
for every $y\in\mathcal{Y}$ and $z\in\text{\ensuremath{\mathcal{Z}}}$.

(ii) $h_{z}(X,\,y)$ is convex in $y\in\mathcal{Y}$ and concave in
$z\in\mathcal{Z}$, for all $X\in\mathcal{L}$. 

(iii) Any selection of subgradients of $h_{z}(X,\,y)$ with respect
to $z\in\mathcal{Z}$ and $y\in\mathcal{Y}$ is Borel measurable and
uniformly bounded for all $X\in\mathcal{L}$.

(iv) $h_{z}$ is Lipschitz continuous on $\mathcal{L}\times\mathcal{Y}$
with constant $K_{G}-1$ for all $z\in\mathcal{Z}$.\\
Then,
\begin{equation}
G(X,\,y,\,z)=y+h_{z}(X,\,y),\,y\in\mathcal{Y},\,z\in\mathcal{Z},\label{Construction G}
\end{equation}
satisfies Assumption \ref{Assumption 3.3}. Further, formulation (\ref{Saddle-2})
with the choice of (\ref{Construction G}) is a convex risk measure
satisfying axioms (A1)-(A3).
\end{thm}

We now detail several applications of Theorem \ref{Theorem 3.2-1}.
\begin{example}
\textbf{\label{Example 3.3-1} Optimized certainty equivalent (OCE},
see \cite{Ben-Tal2007}). Define $\mathcal{Y}=[\eta_{\min},\,\eta_{\max}]$
($\mathcal{Z}$ is a singleton). First, we construct
CVaR by choosing:
\begin{equation}
h_{z}(X,\,y)=-\frac{1}{(1-\alpha)}\max\{X-y,\,0\},\,\alpha\in[0,\,1],\,\forall z\in\mathcal{Z}.\label{piecewise}
\end{equation}
We then obtain: 
\[
G(X,\,y,\,z)=y+(1-\alpha)^{-1}\mathbb{E}_{P}\left[\max\{X-y,\,0\}\right],\,\forall z\in\mathcal{Z}, 
\]
and
\begin{equation}
\textrm{CVaR}_{\alpha}(X):=\min_{\eta\in[\eta_{\min},\,\eta_{\max}]} \mathbb{E}[G(X,\,y,\,z)] = \min_{\eta\in[\eta_{\min},\,\eta_{\max}]}\mathbb{E}[y+h_{z}(X,\,y)] ,\,\forall z\in\mathcal{Z}. \label{CVaR-1}
\end{equation} 
We can generalize CVaR to OCE by substituting a general utility function
in place of (\ref{piecewise}). Define $u:\,\mathbb{R}\rightarrow[0,\,1]$
to be a concave utility function that is differentiable with bounded
derivative $\nabla u(\cdot)$. In this case, $h_{z}(X,\,y)=u(y-X)$, for all $z\in\mathcal{Z}$, and the function $G$ is:
\[
G(X,\,y,\,z)=y-\mathbb{E}_{P}[u(y-X)],\,\forall z\in\mathcal{Z}.
\]
\end{example}

\begin{example}
\textbf{\label{Example 3.3} Absolute semi-deviation} is a type of
mean-risk model. The absolute semi-deviation is $\rho_{AS}(X):=\mathbb{E}[X]+\iota\,\mathbb{E}\left[\left(X-\mathbb{E}([X]\right)_{+}\right]$
for the weight coefficient $\iota\in[0,\,1]$ (see \cite{Shapiro2009}).
Define $\mathcal{Y}=[\eta_{\min},\,\eta_{\max}]$, and $\mathcal{Z}=[0,\,1]$.
By \cite[Chapter 6.5.2]{Shapiro2009}, we have:
$h_{z}(X,\,y)=(1-\iota\,z)X+\iota\,(X-y)_{+}+\iota\,z-1$,
\[
G(X,\,y,\,z)=X+\iota\,(X-y)_{+}+\iota\,z\,(y-X),
\]
and
\[
\rho_{\textrm{AS}}(X):=\min_{y\in[\eta_{\min},\,\eta_{\max}]}\max_{z\in[0,\,1]}\mathbb{E}[G(X,\,y,\,z)] = \min_{y\in[\eta_{\min},\,\eta_{\max}]}\max_{z\in[0,\,1]}\mathbb{E}[y+h_{z}(X,\,y)].
\]
\end{example}

\begin{example}
\label{Example 3.2} \textbf{The functionally coherent risk measure}
(see \cite{Noyan2013,Noyan2015}) is a finite version of the Kusuoka
representation (see e.g. \cite{Shapiro2013}), which is the weighted average multiple CVaR in terms of their confidence levels. Given a range of confidence
levels $[0,\,1)$ with $\{\alpha_{i}\}_{i=0}^{m}\subset[0,\,1)$ and
$0\leq\alpha_{0}<\alpha_{1}<\cdot\cdot\cdot<\alpha_{m}<1$, we define
$\mathcal{P}(\{\alpha_{i}\}_{i=1}^{m})$ to be the set of probability
distributions on $\{\alpha_{i}\}_{i=1}^{m}$, and we let $\mathfrak{M}$
be a closed convex subset of $\mathcal{P}(\{\alpha_{i}\}_{i=1}^{m})$. In this case, we let $z=(z_{1},\,...,\,z_{m})\in\mathbb{R}^{m}$,
$y=(y_{1},...,y_{m})\in\mathbb{R}^{m}$, and $h_{z}(X,\,y)=\sum_{i=1}^{m}z_{i}(1-\alpha_{i})^{-1}\max\{X-y,\,0\}$,
$\mathcal{Y}=[\eta_{\min},\,\eta_{\max}]^{m}$, and $\mathcal{Z}=\mathfrak{M}$,
and
\[
G(X,\,y,\,z)=\sum_{i=1}^{m}z_{i}\left\{ y_{i}+(1-\alpha_{i})^{-1}\mathbb{E}_{P}\left[\max\{X-y_{i},\,0\}\right]\right\} .
\]
We then obtain
\begin{align}
\rho_{\textrm{KS}}(X) & :=\min_{y\in[\eta_{\min},\,\eta_{\max}]^{m}}\max_{z\in\mathfrak{M}} \mathbb{E}[ G(X,\,y,\,z)] = \min_{y\in[\eta_{\min},\,\eta_{\max}]^{m}}\max_{z\in\mathfrak{M}} \mathbb{E}[y+h_{z}(X,\,y)]. \label{Kusuoka}
\end{align}
\end{example}

Each instance of $G$ constructed in Examples
\ref{Example 3.3-1}, \ref{Example 3.2}, and \ref{Example 3.3},
satisfies parts (i)-(iv) in Assumption \ref{Assumption 3.3}. 

\section{Risk-aware Q-learning algorithm}

In this section, we introduce our\emph{ }'Risk-aware $Q$-learning'
(RaQL) algorithm. RaQL is an asynchronous off-policy learning algorithm with an inner and outer loop structure. 
It uses stochastic approximation in the inner loop for risk estimation and $Q$-learning in the outer loop for computing the optimal risk-aware policy. The ``off-policy'' characteristic means that the policy for exploring new states (denoted $\bar{\pi}$) and the
policy $\pi$ from the $Q$-value updates are different. The ``asynchronous'' characteristic means that the step-size rule of the algorithm ensures that only a single state-action pair is updated when it is observed and sends the step-size to zero whenever
a state-action pair is not visited.

\subsection{Algorithm description}

Let $\mathcal{V}\subset\mathbb{R}^{|\mathbb{S}|}$ be the space of
value functions on $\mathbb{S}$ equipped with the supremum norm $\|v\|_{\infty}:=\max_{s\in\mathbb{S}}|v(s)|$.
Under Assumption \ref{Assumption 2.1}, we have $\|v\|_{\infty}\le V_{\max}=C_{\max}/(1-\gamma)$ for all
$v\in\mathcal{V}$. The risk-aware Bellman operator $\mathcal{T}:\,\mathcal{V}\rightarrow\mathcal{V}$
corresponding to the MDP (\ref{Risk MDP 2}) is
\begin{equation}
\left[\mathcal{T}\,v\right]\left(s\right):=\min_{a\in\mathbb{A}}\left\{ c\left(s,\,a\right)+\gamma\,\rho\left(v\left(s^{\prime}\right)\right)\right\} ,\,\forall s\in\mathbb{S},\label{Contraction}
\end{equation}
where $s^{\prime}$ is the random next state following the transition kernel $P(\cdot|s,\,a)$. By \cite[Theorem 4]{Ruszczynski2010}
and \cite[Theorem 5.5]{Shen2013}, $\mathcal{T}$ is a contraction
with respect to the supremum norm and Problem (\ref{Risk MDP 2})
has an optimal value function $v^{\ast}$ satisfying $v^{\ast}=\mathcal{\mathcal{T}}v^{\ast}$.
The following Proposition \ref{Proposition 4.1} demonstrates why
$\mathcal{T}$ is a contraction when $\rho$ is a convex risk measure.
\begin{prop}
\label{Proposition 4.1} Suppose $\rho$ is a convex risk measure,
then
\[
\|\mathcal{T}\,v_{1}-\mathcal{T}\,v_{2}\|_{\infty}\leq\gamma\,\|v_{1}-v_{2}\|_{\infty},
\]
 for all $v_{1},\,v_{2}\in\mathcal{V}$.
\end{prop}

\begin{proof}
By \cite{follmer2002convex,Ruszczynski2006,Brown2012}, any convex
risk measure $\rho$ can be represented as
\begin{equation}
\rho(X)=\sup_{P\in \mathcal{P}(\Omega)}\left\{ \mathbb{E}_{P}[X]-\mu(P)\right\} ,\label{Convex risk-1-1}
\end{equation}
where $\mu$ is a convex function satisfying $\inf_{P\in \mathcal{P}(\Omega)}\mu(P)=0$,
and $P(\Omega)$ is the set of probability distributions on $\left(\Omega,\,\mathcal{F}\right)$.
Then, since $\rho$ is convex risk measure,
\begin{align*}
\|\mathcal{T}\,v_{1}-\mathcal{T}\,v_{2}\|_{\infty}\leq & \gamma\,\left|\sup_{P\in \mathcal{P}(\Omega)}\left\{ \mathbb{E}_{P}[v_{1}]-\mu(P)\right\} -\sup_{P\in \mathcal{P}(\Omega)}\left\{ \mathbb{E}_{P}[v_{2}]-\mu(P)\right\} \right|\\
\leq & \gamma\,\left|\sup_{P\in\mathcal{P} (\Omega)}\mathbb{E}_{P}[v_{1}-v_{2}]\right|\leq\gamma\,\sup_{P\in \mathcal{P}(\Omega)}\mathbb{E}_{P}\left|v_{1}-v_{2}\right|\leq\gamma\,\|v_{1}-v_{2}\|_{\infty},
\end{align*}
since $\mathbb{E}_{P^{\prime}}\left|v_{1}-v_{2}\right|\leq\|v_{1}-v_{2}\|_{\infty}$
for any $P\in\mathcal{P}(\Omega)$.
\end{proof}
For the dynamic setting, we now introduce the risk measure (\ref{Saddle-2})
for each state-action pair. Given the current state $s\in\mathbb{S}$
and action $a\in\mathbb{A}$, the risk for the value of the next state
$s^{\prime}\in\mathcal{\mathbb{S}}$ is defined to be: 
\begin{equation}
\mathcal{R}_{(s,\,a)}^{G}(v(s^{\prime})):=\min_{y\in\mathcal{Y}}\max_{z\in\mathcal{Z}}\mathbb{E}_{s^{\prime}\sim P(\cdot|s,\,a)}\left[G\left(v(s^{\prime}),\,y,\,z\right)\right],\label{Minimax Multiple}
\end{equation}
where the expectation is with respect to the transition kernel $P(\cdot|s,\,a)$.
Throughout the remainder of this paper, we assume that $\mathcal{R}_{(s,\,a)}^{G}$
is a convex risk measure satisfying axioms (A1)-(A3) for all $(s,\,a)\in\mathbb{K}$,
which means that $G$ may be constructed from Theorem \ref{Theorem 3.2-1}.
For simpler notation, we just take $G$ in (\ref{Minimax Multiple})
to be the same for all state-action pairs $\left(s,\,a\right)\in\mathbb{K}$.
We also assume that the $G$ in (\ref{Minimax Multiple}) satisfies
Assumption \ref{Assumption 3.3}. The corresponding risk-aware Bellman
operator is then $\mathcal{T}_{G}\text{: }\mathcal{\mathcal{V}}\rightarrow\mathcal{V}$
defined by
\begin{equation}
\left[\mathcal{T}_{G}\,v\right]\left(s\right):=\min_{a\in\mathbb{A}}\left\{ c\left(s,\,a\right)+\gamma\,\mathcal{R}_{(s,\,a)}^{G}\left(v\left(s^{\prime}\right)\right)\right\} ,\,\forall s\in\mathbb{S}.\label{Dynamic programming 1}
\end{equation}
Since $\mathcal{\mathcal{T}}_{G}$ is a\emph{ }contraction operator,
Problem (\ref{Risk MDP 2}) has an optimal value function $v^{\ast}$
satisfying $v^{\ast}=\mathcal{\mathcal{T}}_{G}v^{\ast}$. Additionally,
based on \cite[Theorem 4]{Ruszczynski2010} and \cite[Theorem 5.5]{Shen2013},
Problem (\ref{Risk MDP 2}) has a stationary optimal policy $\pi^{\ast}\in\Pi$
which is greedy with respect to $v^{*}$, i.e.

\[
\pi^{\ast}(s)\in\arg\min_{a\in\mathbb{A}}\left\{ c\left(s,\,a\right)+\gamma\,\mathcal{R}_{(s,\,a)}^{G}\left(v^{\ast}\left(s^{\prime}\right)\right)\right\} ,\,\forall s\in\mathbb{S}.
\]

Now, based on \cite{Watkins1992} and \cite[Theorem 1]{Shen2014},
we define the risk-aware $Q$-value to be:
\begin{equation}
Q(s,\,a):=c(s,\,a)+\gamma\mathcal{R}_{(s,\,a)}^{G}\left(\min_{a^{\prime}\in\mathbb{A}}Q(s^{\prime},\,a^{\prime})\right),\,\forall(s,\,a)\in\mathbb{K},\label{Q-risk O}
\end{equation}
and the optimal risk-aware $Q$-value, denoted as $Q^{\ast}$, to
be:
\begin{equation}
Q^{\ast}(s,\,a):=c(s,\,a)+\gamma\,\mathcal{R}_{(s,\,a)}^{G}\left(\max_{a^{\prime}\in\mathcal{\mathbb{A}}}Q^{\ast}(s^{\prime},\,a^{\prime})\right),\,\forall(s,\,a)\in\mathbb{K}.\label{Q-risk}
\end{equation}

The procedure of RaQL is presented as Algorithm 1 (we provide the
pseudo code in Algorithm 1 and later give the detailed descriptions
of each step). RaQL is an asynchronous algorithm based on two loops:
an outer loop (with $N$ iterations) and an inner loop (with $T$
iterations). In Algorithm 1, we let $Q_{t}^{n}(s,\,a)$ be the $Q$-value
of state-action pair $(s,\,a)\in\mathbb{K}$ w.r.t. iteration $n\leq N$
and $t\leq T$. Define $\tau_{\ast}(\cdot)$ to be a deterministic
function with $\tau_{\ast}(n)\in[1,\,n]$ for all $n\leq N$ satisfying
the same conditions as in \cite[Algorithm 2.1]{Nemirovski2005}, and
define $H_{\mathcal{Y}}$ and $H_{\mathcal{Z}}$ to be any positive
constants. Here we use $(y_{t}^{n}(s,\,a),\,z_{t}^{n}(s,\,a))$ to
denote the approximate saddle-point of Problem (\ref{Minimax Multiple})
for $(s,\,a)\in\mathbb{K}$ for all $n\leq N$ and $t\leq T$. The
step-sizes are $\theta_{k}^{n}(s,\,a)$ (outer loop) and $\lambda_{t,\alpha}$
(inner loop), and the exploration policy is $\bar{\pi}$. 

Define the collection of state-action pairs $\mathcal{G}:=\sigma\left\{ (s_{t}^{n},\,a^{n}),\,n\leq N,\,t\leq T\right\} $,
and the filtration is $\mathcal{G}_{t}^{n}=\left\{ \sigma\left\{ (s_{\tau}^{i},\,a_{\tau}^{i}),\,i<n,\,\tau\leq T\right\} \cup\left\{ (s_{\tau}^{n},\,a_{\tau}^{n}),\,\tau\leq t\right\} \right\} $
for all $t\leq T$ and $n\leq N$, with $\mathcal{G}_{t}^{0}=\left\{ \varnothing,\,\Omega^{\prime}\right\} $
for all $t\leq T$. This filtration is nested $\mathcal{G}_{t}^{n}\subseteq\mathcal{G}_{t+1}^{n}$
for all $1\leq t\leq T-1$ and $\mathcal{G}_{T}^{n}\subseteq\mathcal{G}_{0}^{n+1}$,
for all $1\leq n\leq N-1$, and captures the history of the algorithm. 
\begin{defn}
\label{Definition 4.4 } Given $\varepsilon\in(0,\,1)$, $\bar{\pi}$
is an $\varepsilon$-greedy exploration policy that chooses a random
action uniformly with probability $\varepsilon$ and otherwise (with
probability $1-\varepsilon$) greedily chooses the action with minimal
$Q$-value. We denote $a^{\prime}\in \arg\min_{a\in\mathbb{A}}Q_{T}^{n-1}(s,\,a)$
and suppose $\bar{\pi}$ satisfies $\mathbb{P}\left((s_{t}^{n},\,a^{n})=(s,\,a)|\mathcal{G}_{t-1}^{n}\right)=\varepsilon$
and $\mathbb{P}\left((s_{t}^{n},\,a^{n})=(s,\,a^{\prime})|\mathcal{G}_{t-1}^{n}\right)=1-\varepsilon$
for all $(s,\,a)\in\mathbb{K}$, for any $n\leq N,\,t\leq T$. Similarly,
we have $\mathbb{P}\left((s_{1}^{n},\,a^{n})=(s,\,a)|\mathcal{G}_{T}^{n-1}\right)=\varepsilon$,
and $\mathbb{P}\left((s_{t}^{n},\,a^{n})=(s,\,a^{\prime})|\mathcal{G}_{T}^{n-1}\right)=1-\varepsilon$
for all $(s,\,a)\in\mathbb{K}$, for any $n\leq N,\,t\leq T$.
\end{defn}

The exploration policy $\bar{\pi}$ in Definition \ref{Definition 4.4 }
guarantees, by the Extended Borel-Cantelli Lemma in \cite{Breiman1992},
that we will visit every state-action pair infinitely often with probability
one. This balances exploration and exploitation in RaQL more generally,
which helps the algorithm avoid getting stuck at locally optimal policies.
It should be noted that RaQL is an off-policy learning algorithm,
so the policy for exploration i.e. $\bar{\pi}$ and the policy from the 
$Q$-value updates (i.e. $\pi$) are different. 
\begin{assumption}
\label{Assumption 4.5} For all $(s,\,a)\in\mathbb{K}$ and for all
$n\leq N,\,t\leq T$, the step-sizes for the $Q$-value update satisfy:
$\sum_{n=1}^{\infty}\theta_{k}^{n}(s,\,a)=\infty$ and $\sum_{n=1}^{\infty}\theta_{k}^{n}(s,\,a)^{2}<\infty$
for all $t\leq T$ and $(s,\,a)\in\mathbb{K}$ a.s. Let $\#(s,\,a,\,n)$
denote one plus the number of times, until the beginning of iteration
$n$, that the state-action pair $(s,\,a)$ has been visited, and
let $N^{s,a}$ denote the set of outer iterations where action $a$
was performed in state $s$. The step-sizes $\theta_{k}^{n}(s,\,a)$
satisfy $\theta_{k}^{n}(s,\,a):=\frac{1}{[\#(s,a,n)]^{k}}$ if $n\in N^{s,a}$
and $\theta_{k}^{n}(s,\,a)=0$ otherwise.
\end{assumption}

Assumption \ref{Assumption 4.5} sends the step-size to zero whenever
a state-action pair is not visited. This step-size selection ensures
that only a single state-action pair is updated when it is observed,
which reveals the asynchronous nature of the $Q$-learning algorithm
stated in \cite{Even-Dar2004}. We choose $k\in(1/2,\,1]$, where
we call $k=1$ a linear learning rate and $k\in(1/2,\,1)$ a polynomial
learning rate, and step-sizes $\lambda_{t,\alpha}=C\,t^{-\alpha}$
for the risk estimation with $\alpha\in(0,\,1]$ for arbitrary $C>0$.
\begin{algorithm}
\caption{Risk-aware $Q$-learning (RaQL)}

\textbf{Begin} 

\qquad Initialization using \textbf{Step 0};

\qquad \textbf{For $n=1,\,2\,,...,\,N$ do }

\qquad \qquad Update the approximation results using \textbf{Step
1};

\qquad \qquad Observe the current state $s_{1}^{n}$, and choose
an action $a^{n}$ according to exploration strategy $\bar{\pi}$;

\qquad \qquad Observe resulting cost $c$, and next state $s_{2}^{n}$; 

\qquad \qquad \textbf{For $t=1,\,2\,,...,\,T$ do}

\qquad \qquad \qquad Update the risk-aware cost-to-go using \textbf{Step
2};

\qquad \qquad \qquad Do stochastic approximation of $\{Q_{t}^{n}\}$
with respect to $t$ using \textbf{Step 3};

\qquad \qquad \qquad Do stochastic approximation of risk measure by
\textbf{Step 4};

\qquad \qquad \qquad Observe new state $s_{t+2}^{n}$, and set
$s_{t}^{n}=s_{t+1}^{n};$

\qquad \qquad \textbf{end for }

\qquad \textbf{end for}

\qquad \textbf{Return} $Q_{T}^{N}$.

\textbf{end}
\end{algorithm}
The detailed description for each step of RaQL follows:

\textbf{Step 0}: Initialize an approximation for the $Q$-values $Q^{0}(s,\,a)$
for all $(s,\,a)\in\mathbb{K}$; given step-sizes $\theta_{k}^{n},\,\lambda_{t,\alpha}>0$
for $t\leq T$ and $n\leq N$, with learning rates $k$ and $\alpha$;
deterministic function $\tau^{\ast}(\cdot)$;
initialize $\left(y_{t}^{0}(s,\,a),\,z_{t}^{0}(s,\,a)\right)$ for
all $t\leq T$ and $(s,\,a)\in\mathbb{K}$.

\textbf{Step 1}: For all $(s,\,a)\in\mathbb{K}$, set $\left(y_{1}^{n}(s,\,a),\,z_{1}^{n}(s,\,a)\right)=\left(y_{T}^{n-1}(s,\,a),\,z_{T}^{n-1}(s,\,a)\right)$
and $Q_{1}^{n}(s,\,a)=Q_{T}^{n-1}(s,\,a)$. 

\textbf{Step 2}: Compute $v^{n-1}(s_{t+1}^{n})=\min_{a\in\mathcal{\mathbb{A}}}Q_{T}^{n-1}(s_{t+1}^{n},\,a)$.
Compute
\begin{equation}
\hat{q}_{t}^{n}(s_{t}^{n},\,a^{n})=c(s_{t}^{n},\,a^{n})+\text{\ensuremath{\gamma\,}}G\left(v^{n-1}(s_{t+1}^{n}),\,y^{n,t}(s_{t}^{n},\,a^{n}),\,z^{n,t}(s_{t}^{n},\,a^{n})\right),\label{cost-to-go}
\end{equation}
and
\begin{equation}
\left(y^{n,t}(s_{t}^{n},\,a^{n}),\,z^{n,t}(s_{t}^{n},\,a^{n})\right)=\frac{1}{t-\tau_{\ast}(t)+1}\sum_{\tau=\tau_{\ast}(t)}^{t}\left(y_{\tau}^{n}(s_{t}^{n},\,a^{n}),\,z_{\tau}^{n}(s_{t}^{n},\,a^{n})\right).\label{gnt}
\end{equation}
To explain, given iteration $n$, in each iteration $t\leq T$, we observe
a new state $s_{t+1}^{n}$ given current state $s_{t}^{n}$ and action
$a^{n}$, compute the estimated risk-aware cost-to-go $\hat{q}_{t}^{n}$
from one sample in Eq. (\ref{cost-to-go}). Here, we use the $Q-$value $Q_{T}^{n-1}$
at iteration $T$ and we compute $v^{n-1}$ from it as input for Eq.  (\ref{cost-to-go}),
although all the $Q$-values $\{Q_{t}^{n-1}\}^{T}_{t=1}$ are recorded. 

\textbf{Step 3}: For all $(s,\,a)\in\mathbb{K}$, compute
\begin{equation}
Q_{t}^{n}(s,\,a)=\left(1-\text{\ensuremath{\theta}}_{k}^{n}(s,\,a)\right)Q_{T}^{n-1}(s,\,a)+\theta_{k}^{n}(s,\,a)\,\hat{q}_{t}^{n}(s_{t}^{n},\,a^{n}).\label{Q-learning}
\end{equation}
This update is the same as in standard $Q$-learning w.r.t. the outer
loop.

\textbf{Step 4}: Update
\begin{align}
\left(y_{t+1}^{n}(s_{t}^{n},\,a^{n}),\,z_{t+1}^{n}(s_{t}^{n},\,a^{n})\right)= & \Pi_{\mathcal{Y}\times\mathcal{Z}}\left\{ \left(y_{t}^{n}(s_{t}^{n},\,a^{n}),\,z_{t}^{n}(s_{t}^{n},\,a^{n})\right)\right.\nonumber \\
 & \left.-\lambda_{t,\alpha}\psi\left(v^{n-1}(s_{t+1}^{n}),\,y^{n,t}(s_{t}^{n},\,a^{n}),\,z^{n,t}(s_{t}^{n},\,a^{n})\right)\right\} ,\label{SASP1}
\end{align}
where $\Pi_{\mathcal{Y}\times\mathcal{Z}}[(y,\,z)]:=\arg\min_{(y^{\prime},\,z^{\prime})\in\mathcal{Y}\times\mathcal{Z}}\|(y,\,z)-(y^{\prime},\,z^{\prime})\|_{2}$
is the Euclidean projection onto $\mathcal{Y}\times\mathcal{Z}$,
and
\begin{align}
 \psi\left(v^{n-1}(s_{t+1}^{n}),\,y^{n,t}(s_{t}^{n},\,a^{n}),\,z^{n,t}(s_{t}^{n},\,a^{n})\right)=  \left(\begin{array}{c}
H_{\mathcal{Y}}G_{y}(v^{n-1}(s_{t+1}^{n}),\,y^{n,t}(s_{t}^{n},\,a^{n}),\,z^{n,t}(s_{t}^{n},\,a^{n}))\\
-H_{\mathcal{Z}}G_{z}(v^{n-1}(s_{t+1}^{n}),\,y^{n,t}(s_{t}^{n},\,a^{n}),\,z^{n,t}(s_{t}^{n},\,a^{n}))
\end{array}\right).\label{SASP2}
\end{align}

We provide some further remarks on Algorithm 1. 
\begin{enumerate}
\item In Eqs. (\ref{gnt}), (\ref{SASP1}), and (\ref{SASP2}), we use
the stochastic approximation for saddle-point problems (SASP) algorithm
as presented in \cite[Algorithm 2.1]{Nemirovski2005} (the detailed
steps appear in Algorithm 2). In Algorithm 1, we apply and extend
SASP to estimate the risk with respect to each state-action pair,
where the value functions on random next states are the problem input.
Classic stochastic approximation may result in extremely slow convergence
for degenerate objectives (i.e. the objective has a singular Hessian).
However, based on the analysis in \cite{Nemirovski2005}, SASP with
properly chosen $\alpha\in(0,\,1]$ preserves a ``reasonable'' (close
to $O(n^{-1/2})$) convergence rate even when the objective is non-smooth
and/or degenerate. For instance, the Kusuoka representation (\ref{Kusuoka})
is non-smooth and degenerate since the Hessian matrix is singular
with respect to $p\in\mathfrak{M}$. Thus, SASP is more appropriate
for estimation of risk measures.
\begin{algorithm}
\caption{SASP}

\textbf{Step 0. }Input: i.i.d samples $\{x_{t}\}_{t=1}^{\infty}$
of random variable $X$; step-sizes $\lambda_{t}=C\,t^{-\alpha}$
with $\alpha\in(0,\,1]$ for $C>0$; deterministic function
$\tau_{\ast}(\cdot)$; initial $(y_{1},\,z_{1})\in\mathcal{Y}\times\mathcal{Z}$;
positive constants $H_{\mathcal{Y}},\,H_{\mathcal{Z}}$,

\textbf{Step 1.} \textbf{for} $t=1,2,...$ \textbf{do}

\textbf{$\qquad$Step 1a.} Update
\begin{align*}
(y_{t+1},\,z_{t+1}) & =\Pi_{\mathcal{Y}\times\mathcal{Z}}\left[(y_{t},\,z_{t})-\lambda_{n}\psi(x_{t};\,y_{t},\,z_{t})\right],\,t\geq 1.
\end{align*}
The vector $\psi(x;\,y,\,z)\in\mathbb{R}^{d_{1}}\times\mathbb{R}^{d_{2}}$
is
\[
\psi(x;\,y,\,z)=\left(H_{\mathcal{Y}}G_{y}(x,\,y,\,z),\,-H_{\mathcal{Z}}G_{z}(x,\,y,\,z)\right),
\]
and $x$ is any realization of the random variable $X$,

\textbf{$\qquad$Step 1b.} Take the moving average
\[
(y_{t},\,z_{t})=\frac{1}{t-\tau_{\ast}(t)+1}\sum_{\tau=\tau_{\ast}(t)}^{t}(y_{\tau},\,z_{\tau}).
\]
\end{algorithm}
\item The risk estimation and the $Q$-value updates are mutually dependent.
Given iteration $n$, the risk estimation, Step 4, applies SASP to update
the candidate solution of the saddle-point problem for each selected
state-action pair, using the $Q$-value from the previous iteration (i.e.
$Q_{T}^{n-1}$). Given the current state-action pair $(s,\,a)$, neither
the expected value of $G$ in (\ref{Minimax Multiple}), nor the subdifferentials
$\{\partial_{y}\mathbb{E}G\left(v(s^{\prime}),\,y,\,z\right),\,\partial_{z}\mathbb{E}G\left(v(s^{\prime}),\,y,\,z\right)\}$
(the expectation is with respect to the transition kernel), are available.
We assume that at any iteration $t$, for every desired point $\left(y^{n,t}(s_{t}^{n},\,a^{n}),\,z^{n,t}(s_{t}^{n},\,a^{n})\right)$,
one can obtain a biased estimator of the aforementioned subgradients.
These estimates form a realization of the pair of random vectors,
\[
G_{y}(v^{n-1}(s_{t+1}^{n}),\,y^{n,t}(s_{t}^{n},\,a^{n}),\,z^{n,t}(s_{t}^{n},\,a^{n}))\in\mathbb{R}^{d_{1}},
\]
 and
\[
G_{z}(v^{n-1}(s_{t+1}^{n}),\,y^{n,t}(s_{t}^{n},\,a^{n}),\,z^{n,t}(s_{t}^{n},\,a^{n}))\in\mathbb{R}^{d_{2}},
\]
where $\{v^{n-1}(s_{t}^{n})\}_{t=1}^{T}$ is a sequence of independent
identically distributed \textquotedblleft observation noises\textquotedblright{}
according to the underlying transition kernel. In Step 3, the $Q$-value
in the current iteration $\{Q_{t}^{n}\}^{T}_{t=1}$ is updated based
on $Q_{T}^{n-1}$ and the approximate risk-to-go $\hat{q}_{t}^{n}(s_{t}^{n},\,a^{n})$
follows the same update rule in standard $Q$-learning. 
\item We resolve the \emph{``overestimation''} problem (the accumulated
error from poor risk estimation) in reinforcement learning described
in \cite{Hasselt2010,VanHasselt2016} through the special inner-outer
loop structure of RaQL. This phenomenon is not mentioned or resolved
in \cite{Shen2014}, where the iterative procedure is analogous to
standard $Q$-learning because of the special structure of utility-based
shortfall. RaQL reduces the bias by multiple iterations of inner loop to provide
an accurate risk estimate before updating the $Q$-values. Our algorithm
is thus related to ``Repeated updated $Q$-learning'' as\emph{ }proposed
in \cite{Abdallah2016}, which resolves performance degradation when
the algorithm is used in noisy non-stationary environments. Our algorithm
addresses what we refer to as the ``policy-bias'' of the action
value update. Policy-bias appears in $Q$-learning because the value
of an action is only updated when the action is executed. Consequently,
the effective rate of updating an action value directly depends on
the probability of choosing the action for execution. For any state-action
pair $(s_{1}^{n},\,a^{n})$ chosen by $\bar{\pi}$ in the outer loop
w.r.t. $n\leq N$, we perform stochastic approximation to
estimate the risk for state-action pairs with fixed action $a^{n}$
in iterations $t\leq T$. This convention increases the probability
of choosing optimal actions while also getting a more accurate risk
estimate. 
\item Often, the cost function is random rather than deterministic. For
example, in inventory control, for stock $s$, we order quantity $a$,
and then only learn the cost after seeing the random demand. Let $c(s,\,a,\,X)$
denote the random cost, where $X$ is random noise, and assume that
$0\leq c(s,\,a,\,X)\leq C_{\max}$ for all $(s,\,a)\in\mathbb{K}$
a.s. Following the same technique as for standard $Q$-learning from
\cite{Tsitsiklis1994,Powell2007}, we can substitute realizations of the random cost
for deterministic costs in our update rule Eq. (\ref{Q-risk}), and
then compute the risk of the sum of the random cost and the discounted
cost-to-go,
\begin{equation}
Q^{\ast}(s,\,a)=\mathcal{R}_{\left(s,\,a\right)}^{G}\left[c(s,\,a,\,X)+\gamma\,\min_{a\in\mathcal{\mathbb{A}}}Q^{\ast}(s^{\prime},\,a)\right],\label{Random Q}
\end{equation}
for all $(s,\,a)\in\mathbb{K}$. Let $\{x_{t}^{n}\}_{1\leq t\leq T,\,1\leq n\leq N}$ denote a sequence
of independent identically distributed samples of $X$ indexed by
$t$ and $n$, and let $c(s_{t}^{n},\,a^{n},\,x_{t}^{n})$ denote
the cost observed in state $s_{t}^{n}$, for action $a^{n}$, with
noise $x_{t}^{n}$ at iteration $t$ and $n$. In terms of solving
Problem (\ref{Random Q}), we replace the earlier expression (\ref{cost-to-go})
in Algorithm 1 with
\begin{equation}
\hat{q}_{t}^{n}(s_{t}^{n},\,a^{n})=G\left(c(s_{t}^{n},\,a^{n},\,x_{t}^{n})+\text{\ensuremath{\gamma}}\,v^{n-1}(s_{t+1}^{n}),\,y^{n,t}(s_{t}^{n},\,a^{n}),\,z^{n,t}(s_{t}^{n},\,a^{n})\right),\label{new q}
\end{equation}
and replace Step 4 in Algorithm 1 with
\begin{align*}
\left(y_{t+1}^{n}(s_{t}^{n},\,a^{n}),\,z_{t+1}^{n}(s_{t}^{n},\,a^{n})\right)= & \Pi_{\mathcal{Y}\times\mathcal{Z}}\left\{ \left(y_{t}^{n}(s_{t}^{n},\,a^{n}),\,z_{t}^{n}(s_{t}^{n},\,a^{n})\right)\right.\\
 & \left.-\lambda_{t,\alpha}\psi\left(c(s_{t}^{n},\,a^{n},\,x_{t}^{n})+\text{\ensuremath{\gamma}}\,v^{n-1}(s_{t+1}^{n}),\,y^{n,t}(s_{t}^{n},\,a^{n}),\,z^{n,t}(s_{t}^{n},\,a^{n})\right)\right\} ,
\end{align*}
and
\begin{align*}
 & \psi\left(v^{n-1}(s_{t+1}^{n}),\,y^{n,t}(s_{t}^{n},\,a^{n}),\,z^{n,t}(s_{t}^{n},\,a^{n})\right)\\
= & \left(\begin{array}{c}
H_{\mathcal{Y}}G_{y}(c(s_{t}^{n},\,a^{n},\,x_{t}^{n})+\text{\ensuremath{\gamma}}\,v^{n-1}(s_{t+1}^{n}),\,y^{n,t}(s_{t}^{n},\,a^{n}),\,z^{n,t}(s_{t}^{n},\,a^{n}))\\
-H_{\mathcal{Z}}G_{z}(c(s_{t}^{n},\,a^{n},\,x_{t}^{n})+\text{\ensuremath{\gamma}}\,v^{n-1}(s_{t+1}^{n}),\,y^{n,t}(s_{t}^{n},\,a^{n}),\,z^{n,t}(s_{t}^{n},\,a^{n}))
\end{array}\right).
\end{align*}
This random cost variant of Algorithm 1 is also based on repeated
stochastic approximation for risk estimation for a fixed action in
the inner loop, which resolves the overestimation problem caused by
the biased risk estimation.  
\end{enumerate}

\subsection{Main results }

We now state our convergence results for RaQL. 
\begin{thm}
\label{Theorem 4.4} (Almost Sure Convergence) Suppose Assumption
\ref{Assumption 4.5} holds, and fix $T\geq 1$. Let
$\{Q_{T}^{n}\}_{n\geq1}$ be the $Q$-value produced by Algorithm 1. Then $Q^{n}_{T}\rightarrow Q^{\ast}$ as $n\rightarrow \infty$, almost surely. 
\end{thm}

The proof of Theorem \ref{Theorem 4.4} uses techniques from the stochastic
approximation literature \cite{Kushner2003}, \cite{Borkar2000} and
\cite{Borkar2008}, which are applied to reinforcement learning and
$Q$-learning in \cite{Watkins1992,Tsitsiklis1994,Bertsekas1996,Jaakkola1994}.
However, our algorithm differs from risk-neutral $Q$-learning because
it updates $Q$-values as well as estimates risk via stochastic approximation.
The intuition of our proof follows the idea in \cite{Jiang2017} where
multiple ``stochastic approximation instances'' for both $Q$-value
updates and risk estimation are ``pasted'' together. The error in
$Q$-values is captured by the distance of $Q_{t}^{n}$ to the optimal
$Q^{*}$, while the error in risk estimation is captured by the duality
gap of the corresponding stochastic saddle-point problem. We must
account for the interdependence of these two errors in several parts
of our proof.

Next, we present the convergence rate for RaQL for a polynomial learning
rate. We first clarify several important concepts and definitions
that appear in this result. For any $(s,\,a)\in\mathbb{K}$, we define
$(y^{n,\ast}(s,\,a),\,z^{n,\ast}(s,\,a))$ to be a saddle-point of
\[
(y(s,\,a),\,z(s,\,a))\rightarrow\mathbb{E}_{s^{\prime}\sim P(\cdot|s,a)}\left[G\left(v^{n-1}(s^{\prime}),\,y(s,\,a),\,z(s,\,a)\right)\right],
\]
for each $(s,\,a)\in\mathbb{K}$, where $v^{n-1}(s^{\prime}):=\min_{a\in\mathcal{\mathbb{A}}}Q_{T}^{n}(s^{\prime},\,a)$.
Similarly, we define $(y^{\ast}(s,\,a),\,z^{\ast}(s,\,a))$ to be
a saddle-point of
\[
(y(s,\,a),\,z(s,\,a))\rightarrow\mathbb{E}_{s^{\prime}\sim P(\cdot|s,a)}\left[G\left(v^{\ast}(s^{\prime}),\,y(s,\,a),\,z(s,\,a)\right)\right],
\]
for each $(s,\,a)\in\mathbb{K}$, where $v^{\ast}(s^{\prime}):=\min_{a\in\mathcal{\mathbb{A}}}Q^{\ast}(s^{\prime},\,a)$.
We define the Hausdorff distance between sets with respect to the
Euclidean norm based on \cite{Rockafellar1998}. Let $\mathfrak{A}$
and $\mathfrak{B}$ be two non-empty subsets of a metric space $(M,\,\|\cdot\|_{2})$.
We define their Hausdorff distance $\mathfrak{D}_{H}(\mathfrak{A},\,\mathfrak{B}\text{)}$
by
\[
\mathfrak{D}_{H}(\mathfrak{A},\,\mathfrak{B}\text{)}:=\max\left\{ \sup_{A\in\mathfrak{A}}\inf_{B\in\mathfrak{B}}\|A-B\|_{2},\,\sup_{B\in\mathfrak{B}}\inf_{A\in\mathfrak{A}}\|A-B\|_{2}\right\} .
\]
Let
\[
\mathcal{S}_{1}^{n,t}:=\{(\partial G_{y}(v^{n-1},\,y^{n,t},\,z^{n,t}),\,\partial G_{z}(v^{n-1},\,y^{n,t},\,z^{n,t}))\},
\]
and
\[
\mathcal{S}_{2}^{n,t}:=\{(\partial G_{y}(v^{\ast},\,y^{n,t},\,z^{n,t}),\,\partial G_{z}(v^{\ast},\,y^{n,t},\,z^{n,t}))\},
\]
be the two subdifferentials of $G$ with respect to $v^{n-1}$
and $v^{\ast}$, given $(y^{n,t},\,z^{n,t})$. The results of the
following lemmas appear in our main convergence rate result. First,
Lemma \ref{Lemma 5.5} bounds $\mathfrak{D}_{H}(\mathcal{S}_{1}^{n,t},\,\mathcal{S}_{2}^{n,t})$
with respect to $\|Q_{T}^{n-1}-Q^{\ast}\|_{2}$. 
\begin{lem}
\label{Lemma 5.5}\cite{Attouch1993}\cite[Theorem 4.1]{Penot1993}
Suppose Assumption \ref{Assumption 3.3}(ii) holds, then there exist
$K_{\psi}^{(1)},\,K_{\psi}^{(2)}>0$, such that
\begin{align}
\mathfrak{D}_{H}(\mathcal{S}_{1}^{n,t},\,\mathcal{S}_{2}^{n,t})\leq K_{\psi}^{(1)}\|Q_{T}^{n-1}-Q^{\ast}\|_{2}+K_{\psi}^{(2)}\sqrt{\|Q_{T}^{n-1}-Q^{\ast}\|_{2}},\label{HS modulus}
\end{align}
for all $n\leq N$ and $t\leq T$. 
\end{lem}

The next result pertains to the modulus of stability of the saddle-points for the estimated risk measure w.r.t. each $n\leq N$. 
\begin{lem}
\label{Lemma 4.7}\cite[Theorem 3.1]{Terazono2015}\cite[Proposition 3.1]{Levy2000}
Suppose Assumption \ref{Assumption 3.3} holds, then there exists
$K_{S}>0$ such that for all $n\leq N$ we have
\begin{align*}
 & \|(y^{\ast},\,z^{\ast})-(y^{n,\ast},\,z^{n,\ast})\|_{2}\\
\leq & K_{S}\left\Vert \mathbb{E}_{s^{\prime}\sim P(\cdot|s,a)}\left[G\left(v^{n-1}(s^{\prime}),\,y^{\ast}(s,\,a),\,z^{\ast}(s,\,a)\right)\right]\right.\\
 & -\left.\mathbb{E}_{s^{\prime}\sim P(\cdot|s,a)}\left[G\left(v^{n-1}(s^{\prime}),\,y^{n,\ast}(s,\,a),\,z^{n,\ast}(s,\,a)\right)\right]\right\Vert _{2}.
\end{align*}
\end{lem}

\begin{thm}
\label{Theorem 4.5-1} (High Probability Convergence Rate) Suppose
Assumption \ref{Assumption 4.5} holds, and choose $\tilde{\varepsilon}>0$
and $\delta\in(0,\,1)$. For a polynomial learning rate (i.e., $k\in(1/2,\,1)$),
there exist $0<\kappa<1/C\,K_{\psi}^{(1)}$ and
\begin{equation}
\beta_{T}:=\frac{K_{G}}{2}\left\{ 1-\gamma-\sqrt{\frac{C(\tau_{\ast}(T))^{-\alpha}}{\kappa(1-C(\tau_{\ast}(T))^{-\alpha}K_{\psi}^{(1)}\kappa)}}-K_{G}K_{S}\right\},\label{Beta-1}
\end{equation}
such that we have $\|Q_{T}^{N}-Q^{\ast}\|_{2}\leq\tilde{\varepsilon}$ with probability at least $1-\delta$, for $N$ and $T$ satisfying:
\begin{equation}
(\tau_{\ast}(T))^{-\alpha}\leq\frac{(1-\gamma-K_{G}K_{S})^{2}\kappa/C}{1+K_{\psi}^{(1)}(1-\gamma-K_{G}K_{S})^{2}\kappa^{2}},\label{C1}
\end{equation}
\begin{equation}
(\tau_{\ast}(T))^{-\alpha}\geq\frac{\left[K_{G}-K_{G}(2\gamma+K_{G}K_{S})-2\right]^{2}\kappa}{C\left\{ K_{G}+\left[K_{G}-K_{G}(2\gamma+K_{G}K_{S})-2\right]^{2}K_{\psi}^{(1)}\kappa^{2}\right\} },\label{C2}
\end{equation}
and
\begin{equation}
N=\Omega\left(\left(\frac{V_{\max}^{2}|\mathbb{S}||\mathbb{A}|\ln(V_{\max}\left(|\mathbb{S}||\mathbb{A}|\right)^{3/2}/[\delta\beta_{T}\tilde{\varepsilon}(1-\varepsilon)]}{\beta_{T}{}^{2}\tilde{\varepsilon}^{2}(1-\varepsilon)^{1+3k}}\right)^{1/k}+\left(\frac{1}{(1-\varepsilon)\beta_{T}}\ln\frac{V_{\max}\sqrt{|\mathbb{S}||\mathbb{A}|}}{\tilde{\varepsilon}}\right)^{\frac{1}{1-k}}\right).\label{Lower bound outer iteration}
\end{equation}
\end{thm}

\begin{rem}
To interpret the bound (\ref{Lower bound outer iteration}), we first
consider its dependence on $\tilde{\varepsilon}$. This dependence
gives us the bound $\Omega((\ln(1/\tilde{\varepsilon})/\tilde{\varepsilon}^{2})^{1/k}+(\ln(1/\tilde{\varepsilon}))^{1/(1-k)})$,
which mirrors the bound for classical asynchronous $Q$-learning in
\cite[Theorem 4]{Even-Dar2004}. The lower bound on the number of
outer iterations $N$ (\ref{Lower bound outer iteration}) is decreasing
with $\beta_{T}$. Since the quantity $\beta_{T}$ is increasing with
$T$, the lower bound on $N$ is decreasing with $T$. Consequently,
improving the quality of risk estimation by increasing the number
of inner loops will improve the overall convergence rate of the algorithm.
In addition, the sample complexity will first decrease and then increase
as a function of the learning rate $k$ (which is also observed for
standard $Q$-learning in \cite{Even-Dar2004}). Furthermore, the
sample complexity is directly proportional to the discount factor
$\gamma$, problem size $|\mathbb{S}||\mathbb{A}|$, and bound $V_{\max}$
on the magnitude of the value functions in $\mathcal{V}$. It is inversely
proportional to the Lipschitz constant $K_{G}$ and the modulus of
the Hausdorff distance of the subdifferentials $K_{\psi}^{\text{(1)}}$.
Increasing $\varepsilon$ will also increase the sample complexity,
revealing that there is a tradeoff between avoiding the algorithm
getting stuck at local optima and reducing the overall computational
complexity. The sample complexity also depends on the risk measure, since
different risk measures have different constants $K_{G}$, $K_{\psi}^{(1)}$,
and $K_{S}$.
\end{rem}

We also derive the convergence rate of RaQL for a linear learning
rate (i.e., $k=1$) in Theorem \ref{Theorem 4.6}. Under a linear learning rate, we can obtain convergence rate results in both probability and expectation
as summarized in Theorem \ref{Theorem 4.7}.
\begin{thm}
\label{Theorem 4.6} (High Probability Convergence Rate) Suppose Assumption
\ref{Assumption 4.5} holds, and choose $\tilde{\varepsilon}>0$ and
$\delta\in(0,\,1)$. For a linear learning rate $k=1$, there exists $\beta_{T}$ as described
in (\ref{Beta-1}), such that we have $\|Q_{T}^{N}-Q^{\ast}\|_{2}\leq\tilde{\varepsilon}$ with probability at least $1-\delta$, for $N$ satisfying
\[
N=\Omega\left(\left(\frac{2+\Psi-\varepsilon}{1-\varepsilon}\right)^{\frac{1}{\beta_{T}}\ln\frac{V_{\max}\sqrt{|\mathbb{S}||\mathbb{A}|}}{\tilde{\varepsilon}}}\frac{V_{\max}^{2}|\mathbb{S}||\mathbb{A}|\ln(V_{\max}(|\mathbb{S}||\mathbb{A}|)^{3/2}/[\Psi\delta\beta_{T}\tilde{\varepsilon}(1-\varepsilon)])}{\Psi^{2}\beta_{T}\tilde{\varepsilon}^{2}(1-\varepsilon)^{2}}\right),
\]
where $\Psi$ is any positive constant, and for $T$ satisfying conditions (\ref{C1})
and (\ref{C2}).
\end{thm}

In Theorem \ref{Theorem 4.6}, the positive constant $\Psi$ is used
to bound the duration of iteration $m$, which starts at time $\tau_{m}$,
and ends at time $\tau_{m+1}$. Define $C_{G}$ be the upper bound: 
\begin{equation}
\left\{ 1+\frac{C}{\kappa(1-C\,K_{\psi}^{(1)}\kappa)}+\left[K_{G}(\gamma+K_{S}K_{G})\right]^{2}\right\} \|Q_{T}^{n-1}-Q^{\ast}\|_{2}^{2}\leq C_{G}, \label{CG}
\end{equation}
almost surely. Next, we prove convergence in expectation. We define a function 
\begin{align}
f(t):= & \left[K_{\mathcal{Y}}H_{\mathcal{Y}}^{-1}+K_{\mathcal{Z}}H_{\mathcal{Z}}^{-1}\right]\frac{t^{\alpha}}{C\left(t-\tau_{\ast}(t)+1\right)} \nonumber
\\
& +\frac{(K_{\mathcal{Y}}+K_{\mathcal{Z}})L}{\sqrt{t-\tau_{\ast}(t)+1}} \nonumber
\\
& +C(K_{\mathcal{Y}}+K_{\mathcal{Z}})^{2}L^{2}\left[H_{\mathcal{Y}}K_{\mathcal{Y}}+H_{\mathcal{Z}}K_{\mathcal{Z}}\right]\tau_{\ast}^{-\alpha}(t),  \label{f}
\end{align}
for all integers $1\leq t\leq T$. 
\begin{thm}
\label{Theorem 4.7} (Convergence Rate in Expectation) Suppose Assumption
\ref{Assumption 4.5} holds, and set linear learning rate $k=1$. Given $\tilde{\varepsilon}>0$, we have 
$\mathbb{E}\left[\|Q_{T}^{N}-Q^{\ast}\|_{2}|\mathcal{G}_{T+1}^{N-1}\right]\leq\tilde{\varepsilon}$, for $N$ satisfying
\[
N=\Omega\left(\frac{\max\left\{ (C_{G}+(\gamma\,f(T))^{2})\varepsilon/\left((2-2\gamma K_{G})\varepsilon^{2}-K_{G}(\gamma-K_{S}K_{G}-\varepsilon\right),\,C_{\max}^{2}|\mathbb{S}||\mathbb{A}|\right\} }{\tilde{\varepsilon}}\right),
\]
where $C_{G}$ is defined in the inequality (\ref{CG}) and $f(T)$ is defined in Eq. ({\ref{f}}) by choosing $t=T$. 
\end{thm}

In Theorem \ref{Theorem 4.7}, the function $f(t)$ bounds the duality
gap of the stochastic saddle-point estimation in each iteration $t\leq T$,
for a fixed iteration $n\leq N$.

\section{Proofs of main results}

\subsection{Almost sure convergence}

We now present the proof of Theorem \ref{Theorem 4.4} step by step.
\textbf{}\\
\textbf{}\\
\textbf{Step 1: Bounding} $\|(y^{n,t},\,z^{n,t})-(y^{n,\ast},\,z^{n,\ast})\|_{2}^{2}$,
\textbf{for all} $n\leq N$ \textbf{and} $t\leq T$, \textbf{by a
function of} $\|Q_{T}^{n-1}-Q^{\ast}\|_{2}^{2}$.
\begin{lem}
\label{Lemma 5.5-1} Suppose Assumption \ref{Assumption 4.5} holds,
then there exists $0<\kappa<1/\left(C\,K_{\psi}^{(1)}\right)$ such
that
\begin{equation}
\|(y^{n,t},\,z^{n,t})-(y^{n,\ast},\,z^{n,\ast})\|_{2}^{2}\leq\frac{C(\tau_{\ast}(t))^{-\alpha}}{\kappa(1-C(\tau_{\ast}(t))^{-\alpha}K_{\psi}^{(1)}\kappa)}\|Q_{T}^{n-1}-Q^{\ast}\|_{2}^{2},\label{Risk bound}
\end{equation}
for all $t\leq T,\,n\leq N$.
\end{lem}

\begin{proof}
From Eq.  (\ref{SASP1}) in Step 4 of Algorithm 1, we know
\begin{align}
\|(y_{t+1}^{n},\,z_{t+1}^{n})-(y^{n,\ast},\,z^{n,\ast})\|_{2}^{2}= & \Biggl\Vert\prod_{\mathcal{Y}\times\mathcal{Z}}\left((y_{t}^{n},\,z_{t}^{n})-\lambda_{t,\alpha}\psi(v^{n-1},\,y^{n,t},\,z^{n,t})\right)-\prod_{\mathcal{Y}\times\mathcal{Z}}(y^{n,\ast},\,z^{n,\ast})\Biggr\Vert_{2}^{2}\nonumber \\
\leq & \|(y_{t}^{n},\,z_{t}^{n})-(y^{n,\ast},\,z^{n,\ast})-\lambda_{t,\alpha}\psi(v^{n-1},\,y^{n,t},\,z^{n,t})\|_{2}^{2}\nonumber \\
\leq & \|(y_{t}^{n},\,z_{t}^{n})-(y^{n,\ast},\,z^{n,\ast})\|_{2}^{2}+(H_{\mathcal{Y}}^{2}+H_{\mathcal{Z}}^{2})L^{2}C^{2}t^{-2\alpha}\nonumber \\
 & -2Ct^{-\alpha}\left((y_{t}^{n},\,z_{t}^{n})-(y^{n,\ast},\,z^{n,\ast})\right)^{\top}\psi(v^{n-1},\,y^{n,t},\,z^{n,t}),\label{projection operator-1}
\end{align}
where the first inequality follows from non-expansiveness of the projection
operator and the second inequality holds by Assumption \ref{Assumption 3.3}(iv).
Based on Lemma \ref{Lemma 5.5}, we have
\begin{align*}
 & \|\psi(v^{n-1},\,y^{n,t},\,z^{n,t})-\psi(v^{\ast},\,y^{n,t},\,z^{n,t})\|_{2}\leq  \mathfrak{D}(\mathcal{S}_{1}^{n,t},\,\mathcal{S}_{2}^{n,t})\leq K_{\psi}^{(1)}\|Q_{T}^{n-1}-Q^{\ast}\|_{2}+K_{\psi}^{(2)}\sqrt{\|Q_{T}^{n-1}-Q^{\ast}\|_{2}}.
\end{align*}
Take the sum of the terms $Ct^{-\alpha}\left((y_{t}^{n},\,z_{t}^{n})-(y^{n,\ast},\,z^{n,\ast})\right)^{\top}\psi(v^{n-1},\,y^{n,t},\,z^{n,t})$
from $\tau_{\ast}(t)$ to $t$, divide by $\frac{1}{t-\tau_{\ast}(t)+1}$,
to obtain:
\begin{align*}
 & \frac{1}{t-\tau_{\ast}(t)+1}\sum_{\tau=\tau_{\ast}(t)}^{t}\left[C\tau^{-\alpha}\left((y_{\tau}^{n},\,z_{\tau}^{n})-(y^{n,\ast},\,z^{n,\ast})\right)^{\top}\psi(v^{n-1},\,y^{n,t},\,z^{n,t})\right]\\
\text{\ensuremath{\leq}} & C(\tau_{\ast}(t))^{-\alpha}\left((y^{n,t},\,z^{n,t})-(y^{n,\ast},\,z^{n,\ast})\right)^{\top}\left(\psi(v^{n-1},\,y^{n,t},\,z^{n,t})-\psi(v^{\ast},\,y^{n,t},\,z^{n,t})\right)\\
\leq & C(\tau_{\ast}(t))^{-\alpha}\|(y^{n,t},\,z^{n,t})-(y^{n,\ast},\,z^{n,\ast})\|_{2}\|\psi(v^{n-1},\,y^{n,t},\,z^{n,t})-\psi(v^{\ast},\,y^{n,t},\,z^{n,t})\|_{2}\\
\leq & C(\tau_{\ast}(t))^{-\alpha}\|(y^{n,t},\,z^{n,t})-(y^{n,\ast},\,z^{n,\ast})\|_{2}\left(K_{\psi}^{(1)}\|Q_{T}^{n-1}-Q^{\ast}\|_{2}+K_{\psi}^{(2)}\sqrt{\|Q_{T}^{n-1}-Q^{\ast}\|_{2}}\right),
\end{align*}
where the first inequality holds due to Assumption \ref{Assumption 3.3}(iii).
Using the inequality $2ab\leq a^{2}\kappa+b^{2}/\kappa$ for all $\kappa>0$,
we obtain
\begin{align}
 & -2C(\tau_{\ast}(t))^{-\alpha}\left((y^{n,t},\,z^{n,t})-(y^{n,\ast},\,z^{n,\ast})\right)^{\top}\left(\psi(v^{n-1},\,y^{n,t},\,z^{n,t})-\psi(v^{\ast},\,y^{n,t},\,z^{n,t})\right)\nonumber \\
\geq- & C(\tau_{\ast}(t))^{-\alpha}K_{\psi}^{(1)}\|(y^{n,t},\,z^{n,t})-(y^{n,\ast},\,z^{n,\ast})\|_{2}^{2}\,\kappa-C(\tau_{\ast}(t))^{-\alpha}\|Q_{T}^{n-1}-Q^{\ast}\|_{2}^{2}/\kappa\nonumber \\
- & C(\tau_{\ast}(t))^{-\alpha}K_{\psi}^{(2)}\|(y^{n,t},\,z^{n,t})-(y^{n,\ast},\,z^{n,\ast})\|_{2}\sqrt{\|Q_{T}^{n-1}-Q^{\ast}\|_{2}}.\label{Deduction}
\end{align}
By summing the right hand side of (\ref{projection operator-1}) from
$\tau_{\ast}(t)$ to $t$, dividing by $\frac{1}{t-\tau_{\ast}(t)+1}$,
and combining with (\ref{Deduction}) we obtain
\begin{align}
 & \frac{1}{t-\tau_{\ast}(t)+1}\sum_{\tau=\tau_{\ast}(t)}^{t}\left(\|(y_{\tau}^{n},\,z_{\tau}^{n})-(y^{n,\ast},\,z^{n,\ast})\|_{2}^{2}+(H_{\mathcal{Y}}^{2}+H_{\mathcal{Z}}^{2})L^{2}C^{2}\tau^{-2\alpha}\right)\nonumber \\
 & -2C(\tau_{\ast}(t))^{-\alpha}\left((y^{n,t},\,z^{n,t})-(y^{n,\ast},\,z^{n,\ast})\right)^{\top}\psi(v^{n-1},\,y^{n,t},\,z^{n,t})\nonumber \\
\geq & C(\tau_{\ast}(t))^{-\alpha}\|(y^{n,t},\,z^{n,t})-(y^{n,\ast},\,z^{n,\ast})\|_{2}^{2}-2\left((y^{n,t},\,z^{n,t})-(y^{n,\ast},\,z^{n,\ast})\right)^{\top}\psi(v^{n-1},\,y^{n,t},\,z^{n,t})\nonumber \\
\geq & (1-C(\tau_{\ast}(t))^{-\alpha}K_{\psi}^{(1)}\kappa)\|(y^{n,t},\,z^{n,t})-(y^{n,\ast},\,z^{n,\ast})\|_{2}^{2}\nonumber \\
 & -C(\tau_{\ast}(t))^{-\alpha}\|Q_{T}^{n-1}-Q^{\ast}\|_{2}^{2}/\kappa-C(\tau_{\ast}(t))^{-\alpha}K_{\psi}^{(2)}\|(y^{n,t},\,z^{n,t})-(y^{n,\ast},\,z^{n,\ast})\|_{2}\sqrt{\|Q_{T}^{n-1}-Q^{\ast}\|_{2}}.\label{Deduction 2}
\end{align}
Since the term $(1-C(\tau_{\ast}(t))^{-\alpha}K_{\psi}^{(1)}\kappa)\|(y_{t}^{n},\,z_{t}^{n})$ decreases with $\kappa$, while the term $C(\tau_{\ast}(t))^{-\alpha}\|Q_{T}^{n-1}-Q^{\ast}\|_{2}^{2}/\kappa$ increases with $\kappa$. We further claim that we can choose $\kappa$ satisfying $0<\kappa<1/(C\,K_{\psi}^{(1)})$
such that
\begin{equation}
(1-C(\tau_{\ast}(t))^{-\alpha}K_{\psi}^{(1)}\kappa)\|(y_{t}^{n},\,z_{t}^{n})-(y^{n,\ast},\,z^{n,\ast})\|_{2}^{2}-C(\tau_{\ast}(t))^{-\alpha}\|Q_{T}^{n-1}-Q^{\ast}\|_{2}^{2}/\kappa\leq0,\label{Negative}
\end{equation}
which gives the desired result.
\end{proof}
\textbf{Step 2: Bounding} $\|(y^{\ast},\,z^{\ast})-(y^{n,\ast},\,z^{n,\ast})\|_{2}$,
\textbf{for all $n\leq N$, by a function of} $\|Q_{T}^{n-1}-Q^{\ast}\|_{2}$\textbf{.
}We first present a well known inequality.
\begin{fact}
\label{Fact 5.2} Given two proper functions $f_{1},\,f_{2}:\,\mathcal{X}\rightarrow\mathbb{R}$,
\[
|\max_{x\in\mathcal{X}}f_{1}(x)-\max_{x\in\mathcal{X}}f_{2}(x)|\leq\max_{x\in\mathcal{X}}|f_{1}(x)-f_{2}(x)|
\]
and
\[
|\min_{x\in\mathcal{X}}f_{1}(x)-\min_{x\in\mathcal{X}}f_{2}(x)|\leq\max_{x\in\mathcal{X}}|f_{1}(x)-f_{2}(x)|.
\]
\end{fact}

In particular, Fact \ref{Fact 5.2} implies that
\begin{equation}
|\min_{a}Q_{T}^{n-1}(s,\,a)-\min_{a}Q^{\ast}(s,\,a)|\leq\|Q_{T}^{n-1}-Q^{\ast}\|_{\infty}\leq\|Q_{T}^{n-1}-Q^{\ast}\|_{2},\label{Q-fact}
\end{equation}
for all $s\in\mathbb{S}$. 
\begin{lem}
\label{Saddle convergence} Suppose Assumption \ref{Assumption 3.3}
holds, then, for all $n\leq N$, we have
\begin{equation}
\|(y^{\ast},\,z^{\ast})-(y^{n,\ast},\,z^{n,\ast})\|_{2}\leq K_{S}\,K_{G}\|Q_{T}^{n-1}-Q^{\ast}\|_{2}.\label{Relationship}
\end{equation}
\end{lem}

\begin{proof}
It can be shown that
\begin{align*}
 & \|(y^{\ast},\,z^{\ast})-(y^{n,\ast},\,z^{n,\ast})\|_{2}\\
\leq & K_{S}\left\Vert \mathbb{E}_{s^{\prime}\sim P(\cdot|s,a)}\left[G\left(v^{n-1}(s^{\prime}),\,y^{\ast}(s,\,a),\,z^{\ast}(s,\,a)\right)\right]-\mathbb{E}_{s^{\prime}\sim P(\cdot|s,a)}\left[G\left(v^{n-1}(s^{\prime}),\,y^{n,\ast}(s,\,a),\,z^{n,\ast}(s,\,a)\right)\right]\right\Vert _{2}\\
\leq & K_{S}\max_{z\in\mathcal{Z}}\left\Vert \min_{y\in\mathcal{Y}}\mathbb{E}_{s^{\prime}\sim P(\cdot|s,a)}\left[G\left(v^{n-1}(s^{\prime}),\,y(s,\,a),\,z(s,\,a)\right)\right]-\min_{y\in\mathcal{Y}}\mathbb{E}_{s^{\prime}\sim P(\cdot|s,a)}\left[G\left(v^{\ast}(s^{\prime}),\,y(s,\,a),\,z(s,\,a)\right)\right]\right\Vert _{2}\\
\leq & K_{S}\max_{y\in\mathcal{Y},\,z\in\mathcal{Z}}\left\Vert \mathbb{E}_{s^{\prime}\sim P(\cdot|s,a)}\left[G\left(v^{n-1}(s^{\prime}),\,y(s,\,a),\,z(s,\,a)\right)\right]-\mathbb{E}_{s^{\prime}\sim P(\cdot|s,a)}\left[G\left(v^{\ast}(s^{\prime}),\,y(s,\,a),\,z(s,\,a)\right)\right]\right\Vert _{2}\\
\leq & K_{S}\,K_{G}\|\min_{a}Q_{t}^{n-1}(s,\,a)-\min_{a}Q^{\ast}(s,\,a)\|_{2}\\
\leq & K_{S}\,K_{G}\|Q_{t}^{n-1}-Q^{\ast}\|_{2},
\end{align*}
where the first inequality follows from Lemma \ref{Lemma 4.7}, the
second inequalities holds due to Fact \ref{Fact 5.2}, the
third inequality holds by Lipschitz continuity of $G$ under Assumption
\ref{Assumption 3.3}(ii), fourth inequality holds based on the Lipschitz continuity of function $G$ on $\mathcal{V}\times\mathcal{Y}\times\mathcal{Z}$, and the last inequality holds by inequality (\ref{Q-fact}). 
\end{proof}
\textbf{Step 3: Apply the stochastic approximation convergence theorem.}
This step completes the proof of Theorem \ref{Theorem 4.4}. We introduce
an operator $H:\,\mathcal{V}\times\mathcal{Y}\times\mathcal{Z}\rightarrow\mathbb{R}^{|\mathbb{S}||\mathbb{A}|}$
defined via
\[
H(v,\,y,\,z)(s,\,a):=c(s,\,a)+\gamma G\left(v(s^{\prime}),y(s,\,a),\,z(s,\,a)\right),\,s^{\prime}\sim P(\cdot|s,a),
\]
for all $(s,\,a)\in\mathbb{K}$. Eq. (\ref{Q-risk}) may then be written
as
\begin{equation}
Q^{\ast}(s,\,a)=\mathbb{E}_{s^{\prime}\sim P(\cdot|s,a)}\left[H\left(v^{\ast},\,y^{\ast},\,z^{\ast}\right)\right](s,\,a),\,\forall(s,\,a)\in\mathbb{K}.\label{stable}
\end{equation}
Next we define two stochastic processes, 
\begin{align*}
\epsilon_{t}^{n}(s,\,a) & :=H\left(v^{n-1},\,y^{n,\ast},\,z^{n,\ast}\right)(s,\,a)-H\left(v^{n-1},\,y^{n,t},\,z^{n,t}\right)(s,\,a),\\
\xi_{t}^{n}(s,\,a) & :=H\left(v^{\ast}\,y^{\ast},\,z^{\ast}\right)(s,\,a)-H\left(v^{n-1},\,y^{n,\ast},\,z^{n,\ast}\right)(s,\,a),
\end{align*}
for all $(s,\,a)\in\mathbb{K}$, $t\leq T$, and $n\leq N$, where
$\epsilon_{t}^{n}$ corresponds to the risk estimation error (i.e.
the duality gap of the stochastic saddle-point problem) and $\xi_{t}^{n}$
corresponds to the error between $Q_{T}^{n}$ and $Q^{\ast}$. Therefore,
Step 3 in Algorithm 1 is equivalent to
\begin{align}
 & Q_{t}^{n}(s,\,a)-Q_{T}^{n-1}(s,\,a)\nonumber \\
= & -\text{\ensuremath{\theta}}_{k}^{n}\left[Q_{T}^{n-1}(s,\,a)-Q^{\ast}(s,\,a)+\xi_{t}^{n}(s,\,a)+\epsilon_{t}^{n}(s,\,a)\right.\left.+Q^{\ast}(s,\,a)-H\left(v^{\ast},\,y^{\ast},\,z^{\ast}\right)(s,\,a)\right],\label{Update}
\end{align}
for all $(s,\,a)\in\mathbb{K}$, $t\leq T$, and $n\leq N$. Clearly,
$\mathbb{E}\left[Q^{\ast}(s,\,a)-H\left(v^{\ast},\,y^{\ast},\,z^{\ast}(s,\,a)\right)|\mathcal{G}_{t+1}^{n-1}\right]=0$
for all $(s,\,a)\in\mathbb{K}$ by (\ref{stable}). By Lemma \ref{Lemma 5.5-1},
we know that
\begin{equation}
\mathbb{E}\left[\|\epsilon_{t}^{n}(s,\,a)\|_{2}^{2}|\mathcal{G}_{t+1}^{n-1}\right]\leq\frac{C\,K_{G}}{\kappa(1-C\,K_{\psi}^{(1)}\kappa)}\|Q_{T}^{n-1}-Q^{\ast}\|_{2}^{2},\,\forall(s,\,a)\in\mathbb{K}.\label{Boundedness}
\end{equation}
In particular, inequality (\ref{Boundedness}) holds by setting $t=1$
in (\ref{Risk bound}). Furthermore, inequality (\ref{Boundedness}) shows that
the conditional expectation w.r.t. $\mathcal{G}_{t+1}^{n-1}$ of the
risk estimation error for each state-action pair at each iteration is
always bounded by $\|Q_{T}^{n-1}-Q^{\ast}\|_{2}^{2}$. In addition,
by Lipschitz continuity of $G$, we have
\begin{align}
|\xi_{t}^{n}(s,\,a)|\leq & \gamma\,K_{G}|\min_{a\in\mathbb{A}}Q_{T}^{n-1}(s_{t}^{n},\,a)-\min_{a\in\mathbb{A}}Q^{\ast}(s_{t}^{n},\,a)| +\gamma K_{G}|(y^{\ast}(s,\,a),\,z^{\ast}(s,\,a))-(y^{n,\ast}(s,\,a),\,z^{n,\ast}(s,\,a))|\nonumber \\
\leq & \gamma K_{G}\left[\|Q_{T}^{n-1}-Q^{\ast}\|_{2}+\|(y^{\ast},\,z^{\ast})-(y^{n,\ast},\,z^{n,\ast})\|_{2}\right]\nonumber \\
\leq & \gamma K_{G}(1+K_{G}K_{S})\|Q_{T}^{n-1}-Q^{\ast}\|_{2},\label{Induction}
\end{align}
where the final inequality holds due to Lemma \ref{Saddle convergence}.
Since $G$ is $P$-square summable for every $y\in\mathcal{Y}$ and
$z\in\text{\ensuremath{\mathcal{Z}}}$ in Assumption \ref{Assumption 3.3}(i),
the risk measure (\ref{Minimax Multiple}) is bounded for all $(s,\,a)\in\mathbb{K}$.
Furthermore, based on (\ref{Q-risk O}) and boundedness of the cost
function $c(s,\,a)$ by Assumption \ref{Assumption 2.1}(ii), we have
boundedness of $Q_{t}^{n}$ for all $t\le T,\,n\leq N$. Boundedness
of the $Q$-values together with results (\ref{Boundedness}) and
(\ref{Induction}), along with the equality $\mathbb{E}\left[Q^{\ast}-H\left(v^{\ast},\,y^{\ast},\,z^{\ast}\right)|\mathcal{G}_{t+1}^{n-1}\right]=0$,
mean that the update rule (\ref{Update}) satisfies the conditions
of \cite[Theorem 2.4]{Kushner2003} and \cite[Assumption A2]{Borkar2000}.
We now have the ingredients needed to apply the stochastic approximation
convergence theorem (see e.g. \cite[Theorem 2.4]{Kushner2003} or
\cite[Theorem 2.2]{Borkar2000}) to Eq. (\ref{Update}) to conclude
that
\[
Q_{T}^{n}(s,\,a)\rightarrow Q^{\ast}(s,\,a)\quad a.s.,
\]
for all $(s,\,a)\in\mathbb{K}$ as $n\rightarrow\infty$.
\begin{rem}
In terms of RaQL with random costs, the result of Lemma \ref{Lemma 5.5}
holds for the modified subdifferentials of the function $G$ with
respect to $c(s_{t}^{n},\,a^{n},\,x_{t}^{n})+\gamma\,v^{n-1}(s_{t}^{n},\,a^{n})$
and $c(s_{t}^{n},\,a^{n},\,x_{t}^{n})+\gamma\,v^{\ast}(s_{t}^{n},\,a^{n})$,
based on (\ref{new q}). We can then follow Steps 1 and 2 to bound $\|(y^{n,t},\,z^{n,t})-(y^{n,\ast},\,z^{n,\ast})\|_{2}^{2}$
and $\|(y^{\ast},\,z^{\ast})-(y^{n,\ast},\,z^{n,\ast})\|_{2}$ in
terms of $\|Q_{T}^{n-1}-Q^{\ast}\|_{2}$. Finally, we can again apply
the stochastic approximation convergence theorem \cite[Theorem 2.4]{Kushner2003}
or \cite[Theorem 2.2]{Borkar2000} to prove almost sure convergence
of RaQL for random costs.
\end{rem}

\subsection{Convergence rate}

In this subsection we derive the convergence rate of RaQL for a polynomial
learning rate $k\in(1/2,\,1)$. Our convergence rate proof follows
\cite{Even-Dar2004}. The main idea is to connect the convergence
of RaQL with the convergence of an artificial deterministic sequence,
which has a linear convergence rate that is easier to derive. In other words, the values $\{Q^{n}_{T}\}_{n\geq 1}$ could be bounded by a deterministic sequence almost surely in each iteration. Explicitly,
we will construct $0<\beta_{T}<1$ and an artificial deterministic
sequence $\{D_{m}\}_{m\geq1}$ satisfying $D_{0}=D_{1}=V_{\max}$
and $D_{m+1}=(1-\beta_{T})D_{m}$ for all $m\geq1$. Here we call $m$ as ``epoch''. Clearly, the
sequence $\{D_{m}\}_{m\geq1}$ converges to zero. This sequence also
has the following special property: for every $m\geq1$, there exists
$\tau_{m}$ such that $\|Q_{T}^{n}-Q^{*}\|_{2}\leq D_{m}$ holds for
all $n\geq\tau_{m}$. The duration of epoch $m$ is then $\tau_{m+1}-\tau_{m}$. Subsequently, we show that $Q_{t}^{n}(s,\,a)-Q^{*}(s,\,a)$
is bounded by two simpler stochastic processes $\{Z_{t}^{n,\tau_{m}}(s,\,a)\}_{\tau_{m}\leq n\leq N}$
and $\{Y_{t}^{n,\tau_{m}}(s,\,a)\}_{\tau_{m}\leq n\leq N}$. We then
establish the relationship of $\{Z_{t}^{n,\tau_{m}}(s,\,a)\}_{\tau_{m}\leq n\leq N}$
and $\{Y_{t}^{n,\tau_{m}}(s,\,a)\}_{\tau_{m}\leq n\leq N}$ with $\{D_{m}\}_{m\geq1}$.
In particular, $\{Z_{t}^{l,\tau_{m}}(s,\,a)\}_{l=1,...,n}$ is a martingale
difference sequence so we can derive a high probability bound on $\{Z_{t}^{n,\tau_{m}}(s,\,a)\}_{\tau_{m}\leq n\leq N}$
from the Azuma-Hoeffding inequality. On the other hand, $\{Y_{t}^{n,\tau_{m}}(s,\,a)\}_{\tau_{m}\leq n\leq N}$
captures all of the biased estimation error terms (from the risk estimation
error and the $Q$-value estimation error) in RaQL, which can be bounded
almost surely by a function of $\{D_{m}\}_{m\geq1}$. By combining
these two results, we show that $\|Q_{T}^{N}-Q^{\ast}\|_{2}\leq\tilde{\varepsilon}$
holds with high probability for large enough $N$ and any $T\geq 1$. 

We will verify the existence and provide the explicit forms of: $D_{m}$,
$\beta_{T}$, and $\tau_{m}$ in the upcoming steps.\\
\\
\textbf{Step 1: Constructing two stochastic processes and bounding
$\|Q_{t}^{n}-Q^{*}\|_{2}$ by their sum.} We decompose (\ref{Update})
into two separate stochastic processes $\{Z_{t}^{n,\tau_{m}}(s,\,a)\}_{\tau_{m}\leq n\leq N}$
and $\{Y_{t}^{n,\tau_{m}}(s,\,a)\}_{\tau_{m}\leq n\leq N}$. We define, for a fixed $m\geq 1$ and 
for $n\geq\tau_{m}$ and $t\leq T$, the quantity
\begin{align}
Z_{t}^{n+1,\tau_{m}}(s,\,a):= & (1-\text{\ensuremath{\theta}}_{k}^{n}(s,\,a))Z_{t}^{n,\tau_{m}}(s,\,a)+\text{\ensuremath{\theta}}_{k}^{n}(s,\,a)\zeta(s,\,a),\label{Process W-1}
\end{align}
for all $(s,\,a)\in\mathbb{K}$, where $\zeta(s,\,a):=Q^{\ast}(s,\,a)-H\left(v^{\ast},\,y^{\ast},\,z^{\ast}\right)(s,\,a)$
and $Z_{t}^{\tau_{m},\tau_{m}}=0,$ for all $t\leq T$. We also define, for a fixed $m\geq 1$ and for $n\geq\tau_{m}$ and $t\leq T$, the quantity

\begin{align}
Y_{t}^{n+1,\tau_{m}}(s,\,a):= & (1-\text{\ensuremath{\theta}}_{k}^{n}(s,\,a))Y_{t}^{n,\tau_{m}}(s,\,a)+\text{\ensuremath{\theta}}_{k}^{n}(s,\,a)K_{G}\gamma\left(D_{m}+\|(y^{\ast},\,z^{\ast})-(y^{n,\ast},\,z^{n,\ast})\|_{2}\right)\nonumber \\
 & +\text{\ensuremath{\theta}}_{k}^{n}(s,\,a)K_{G}\|(y^{n,t},\,z^{n,t})-(y^{n,\ast},\,z^{n,\ast})\|_{2},\label{Process Y-1}
\end{align}
for all $(s,\,a)\in\mathbb{K}$, where $Y_{t}^{\tau_{m},\tau_{m}}(s,\,a)=D_{m},$
for all $t\leq T$. The process (\ref{Process W-1}) is a recursion
for the unbiased error terms $\zeta(s,\,a)$, while the process (\ref{Process Y-1})
is a recursion for the biased error terms (e.g. the sum of the $Q$-value
approximation errors and the duality gaps from the risk estimation).
The following Lemma \ref{Lemma 5.5-2} which appears in \cite{Bertsekas1996}
and \cite[Lemma 9]{Even-Dar2004} shows that the almost sure lower and upper bounds for  $Q$-value estimation
error at each iteration by the process (\ref{Process Y-1}) and (\ref{Process W-1}).
\begin{lem}
\label{Lemma 5.5-2} Given the stochastic processes (\ref{Process W-1})
and (\ref{Process Y-1}), and the update rule (\ref{Update}),
\begin{align}
Z_{t}^{n,\tau_{m}}(s,\,a)-Y_{t}^{n,\tau_{m}}(s,\,a)\leq Q_{t}^{n}(s,\,a)-Q^{*}(s,\,a)\leq & Z_{t}^{n,\tau_{m}}(s,\,a)+Y_{t}^{n,\tau_{m}}(s,\,a),\label{Induction-2}
\end{align}
holds for all $t\leq T,\,\tau_{m}\leq n\leq N$, and $(s,\,a)\in\mathbb{K}$.
\end{lem}

\begin{proof}
Suppose both $\epsilon_{t}^{n}(s,\,a)$ and $\xi_{t}^{n}(s,\,a)$
are non-negative for all $(s,\,a)\in\mathbb{K}$. From the right hand
side of Eq. (\ref{Process Y-1}), for all $t\leq T,\,n\leq N$, we have
\begin{align}
 & (1-\text{\ensuremath{\theta}}_{k}^{n}(s,\,a))Y_{t}^{n,\tau_{m}}(s,\,a)+\text{\ensuremath{\theta}}_{k}^{n}(s,\,a)K_{G}\gamma\left(D_{m}+\|(y^{\ast},\,z^{\ast})-(y^{n,\ast},\,z^{n,\ast})\|_{2}\right)\nonumber \\
 & +\text{\ensuremath{\theta}}_{k}^{n}(s,\,a)K_{G}\|(y^{n,t},\,z^{n,t})-(y^{n,\ast},\,z^{n,\ast})\|_{2}\nonumber \\
\geq & (1-\text{\ensuremath{\theta}}_{k}^{n}(s,\,a))Y_{t}^{n,\tau_{m}}(s,\,a)+\text{\ensuremath{\theta}}_{k}^{n}(s,\,a)\left(|\xi_{t}^{n}(s,\,a)|+|\epsilon_{t}^{n}(s,\,a)|\right)\nonumber \\
= & (1-\text{\ensuremath{\theta}}_{k}^{n}(s,\,a))Y_{t}^{n,\tau_{m}}(s,\,a)+\text{\ensuremath{\theta}}_{k}^{n}(s,\,a)\left(\xi_{t}^{n}(s,\,a)+\epsilon_{t}^{n}(s,\,a)\right),\label{New Y}
\end{align}
where this inequality is due to inequality (\ref{Induction}) ($|\xi_{t}^{n}(s,\,a)|\leq\gamma K_{G}\left[\|Q_{T}^{n-1}-Q^{\ast}\|_{2}+\|(y^{\ast},\,z^{\ast})-(y^{n,\ast},\,z^{n,\ast})\|_{2}\right]$),
the definition of $\epsilon_{t}^{n}(s,\,a)$, and Lipschitz continuity
of $G$ from Assumption \ref{Assumption 3.3}(ii) ($|\epsilon_{t}^{n}(s,\,a)|\leq K_{G}\|(y^{n,t},\,z^{n,t})-(y^{n,\ast},\,z^{n,\ast})\|_{2}$).
Combining inequality (\ref{New Y}) with Eqs. (\ref{Process Y-1}) and (\ref{Process W-1}),
for all $t\leq T,\,n\leq N$, and $(s,\,a)\in\mathbb{K}$, we have
\begin{align}
 & Z_{t}^{n+1,\tau_{m}}(s,\,a)+Y_{t}^{n+1,\tau_{m}}(s,\,a)\nonumber \\
\geq & (1-\text{\ensuremath{\theta}}_{k}^{n}(s,\,a))\left(Z_{t}^{n,\tau_{m}}(s,\,a)+Y_{t}^{n,\tau_{m}}(s,\,a)\right)\nonumber \\
 & +\text{\ensuremath{\theta}}_{k}^{n}(s,\,a)\left[\xi_{t}^{n}(s,\,a)+\epsilon_{t}^{n}(s,\,a)+Q^{\ast}(s,\,a)-H\left(v^{\ast},\,y^{\ast},\,z^{\ast}\right)(s,\,a)\right].\label{Induction-1}
\end{align}
We now use induction on $n$ to show that the right-hand side
in inequalities (\ref{Induction-2}) holds. By setting the base case to be $n=\tau_{m}$,
we have, for all $t\leq T$ and $(s,\,a)\in\mathbb{K}$, that
\[
Z_{t}^{\tau_{m},\tau_{m}}(s,\,a)+Y_{t}^{\tau_{m},\tau_{m}}(s,\,a)=D_{m}\geq\|Q_{t}^{\tau_{m}}-Q^{*}\|_{2}\geq Q_{t}^{n}(s,\,a)-Q^{*}(s,\,a).
\]
The above equality holds by the definition of $Z_{t}^{n}(s,\,a)$
and $Y_{t}^{n}(s,\,a)$ in Eqs. (\ref{Process W-1}) and (\ref{Process Y-1}),
respectively. Suppose $Z_{t}^{n,\tau_{m}}(s,\,a)+Y_{t}^{n,\tau_{m}}(s,\,a)\geq Q_{t}^{n}(s,\,a)-Q^{*}(s,\,a)$
for all $\tau_{m}\leq n\leq N$, $t\leq T$, and $(s,\,a)\in\mathbb{K}$.
Then, by inequality (\ref{Induction-1}) and Eq. (\ref{Update}), we obtain the right
hand side of the above inequality. Now, suppose both $\epsilon_{t}^{n}(s,\,a)$
and $\xi_{t}^{n}(s,\,a)$ are negative for all $(s,\,a)\in\mathbb{K}$,
then
\begin{align}
 & -(1-\text{\ensuremath{\theta}}_{k}^{n}(s,\,a))Y_{t}^{n,\tau_{m}}(s,\,a)+\text{\ensuremath{\theta}}_{k}^{n}(s,\,a)K_{G}\gamma\left(-D_{m}-\|(y^{\ast},\,z^{\ast})-(y^{n,\ast},\,z^{n,\ast})\|_{2}\right)\nonumber \\
 & -\text{\ensuremath{\theta}}_{k}^{n}(s,\,a)K_{G}\|(y^{n,t},\,z^{n,t})-(y^{n,\ast},\,z^{n,\ast})\|_{2}\nonumber \\
\leq & -(1-\text{\ensuremath{\theta}}_{k}^{n}(s,\,a))Y_{t}^{n,\tau_{m}}(s,\,a)+\text{\ensuremath{\theta}}_{k}^{n}(s,\,a)(\xi_{t}^{n}(s,\,a)+\epsilon_{t}^{n}(s,\,a)),\label{New Y-2}
\end{align}
and so we obtain the left hand inequality in inequalities (\ref{Induction-2})
following the same reasoning. Finally, when $\epsilon_{t}^{n}(s,\,a)$
and $\xi_{t}^{n}(s,\,a)$ have different signs for all $(s,\,a)\in\mathbb{K}$,
we can show that inequalities (\ref{New Y}) and (\ref{New Y-2})
hold. Thus, following the same inductive reasoning, we can show that
both inequalities in (\ref{Induction-2}) hold.
\end{proof}
Our main focus is on deriving a high probability bound for the convergence
rate of $\|Q_{T}^{n}-Q^{*}\|_{2}$. By Lemma (\ref{Lemma 5.5-2}),
this goal is equivalent to bounding the sum and the difference of
the stochastic processes $\{Z_{t}^{n,\tau_{m}}(s,\,a)\}_{\tau_{m}\leq n\leq N}$
and $\{Y_{t}^{n,\tau_{m}}(s,\,a)\}_{\tau_{m}\leq n\leq N}$. In the
following steps of the proof, we first derive almost sure bounds on $\{Y_{T}^{n,\tau_{m}}\}_{\tau_{m}\leq n\leq N}$.\textbf{}\\
\textbf{}\\
\textbf{Step 2: Bounding} $\{Y_{T}^{n,\tau_{m}}\}_{\tau_{m}\leq n\leq N}$\textbf{
and selecting} $\beta_{T}$\textbf{.} The next lemma provides an almost
sure upper bound on the stochastic process $\{Y_{T}^{n,\tau_{m}}\}_{\tau_{m}\leq n\leq N}$.
Furthermore, this lemma shows that the duration of epoch $m$, which
starts at time $\tau_{m}$ and ends at time $\tau_{m+1}$, is bounded
by $(\tau_{m})^{k}$, and it also provides an explicit selection of
$\beta_{T}$.
\begin{lem}
\label{Lemma 5.8-1} Given any $m\geq1$, assume that for all $n\geq\tau_{m}$, we have $Y_{T}^{n,\tau_{m}}(s,\,a)\leq D_{m}$. Then, for any $n\geq\tau_{m}+(\tau_{m})^{k}=\tau_{m+1}$
we have
\begin{align*}
Y_{T}^{n,\tau_{m}}(s,\,a)\leq & K_{G}D_{m}\left(\gamma+\sqrt{\frac{C(\tau_{\ast}(T))^{-\alpha}}{\kappa(1-C(\tau_{\ast}(T))^{-\alpha}K_{\psi}^{(1)}\kappa)}}+K_{G}K_{S}\right)+\frac{2}{e}\beta_{T}D_{m},
\end{align*}
for all $(s,\,a)\in\mathbb{K}$, where $\beta_{T}$ is given in Eq. (\ref{Beta-1})
with $T$ satisfies conditions (\ref{C1}) and (\ref{C2}).
\end{lem}

\begin{proof}
Based on the proof in \cite[Lemma 27]{Even-Dar2004}, we combine
the convergence rate results from Lemmas \ref{Lemma 5.5-1} and \ref{Saddle convergence} into the definition of the process $Y_{T}^{n+1,\tau_{m}}$
as described in Eq. (\ref{Process Y-1}). We then obtain
\[
Y_{T}^{\tau_{m},\tau_{m}}(s,\,a)=K_{G}D_{m}\left(\gamma+\sqrt{\frac{C(\tau_{\ast}(T))^{-\alpha}}{\kappa(1-C(\tau_{\ast}(T))^{-\alpha}K_{\psi}^{(1)}\kappa)}}+K_{G}K_{S}\right)+g_{\tau_{m}},
\]
for all $(s,\,a)\in\mathbb{K}$ and $m\geq1$, where
\[
g_{\tau_{m}}:=K_{G}\left\{ 1-\gamma-\sqrt{\frac{C(\tau_{\ast}(T))^{-\alpha}}{\kappa(1-C(\tau_{\ast}(T))^{-\alpha}K_{\psi}^{(1)}\kappa)}}-K_{G}K_{S}\right\} D_{m},
\]
for all $(s,\,a)\in\mathbb{K}$ and $m\geq1$. We may then write
\begin{align*}
Y_{T}^{n,\tau_{m}}(s,\,a) & =K_{G}D_{m}\left(1-\gamma-\sqrt{\frac{C(\tau_{\ast}(T))^{-\alpha}}{\kappa(1-C(\tau_{\ast}(T))^{-\alpha}K_{\psi}^{(1)}\kappa)}}-K_{G}K_{S}\right)+(1-\theta_{k}^{n})g_{n},
\end{align*}
for all $(s,\,a)\in\mathbb{K}$, $n\leq N$ and $m\geq1$. Since the step sizes
$\theta_{k}^{n}$ are monotonically decreasing in $n$, we can rewrite $g_{n}$
as
\begin{align}
g_{n} & =K_{G}\left\{ 1-\gamma-\sqrt{\frac{C(\tau_{\ast}(T))^{-\alpha}}{\kappa(1-C(\tau_{\ast}(T))^{-\alpha}K_{\psi}^{(1)}\kappa)}}-K_{G}K_{S}\right\} D_{m}\prod_{l=1}^{n-\tau_{m}}(1-\theta_{k}^{l+\tau_{m}})\nonumber \\
 & \leq2\beta_{T}D_{m}\prod_{l=1}^{n-\tau_{m}}(1-\frac{1}{(l+\tau_{m})^{k}})\nonumber \\
 & \leq2\beta_{T}D_{m}\prod_{l=1}^{n-\tau_{m}}(1-(\frac{1}{\tau_{m}})^{k})^{n-\tau_{m}}\nonumber \\
 & \leq2\beta_{T}D_{m}(1-(\frac{1}{\tau_{m}})^{k})^{(\tau_{m})^{k}}\nonumber \\
 & \leq\frac{2}{e}\beta_{T}D_{m},\label{Difference}
\end{align}
for all $n\leq N$, where the first inequality holds by the choice of $\beta_{T}$ in
Eq. (\ref{Beta-1}).
\end{proof}
\textbf{Step 3: Deriving high probability bound on }$\{|Z_{T}^{n,\tau_{m}}|\}_{\tau_{m}\leq n\leq N}$
\textbf{by the Azuma-Hoeffding inequality. }The following Lemma \ref{Lemma 5.2}
directly follows from the results in \cite[Lemma 28]{Even-Dar2004}.
It shows that $\{Z_{t}^{l,\tau_{m}}(s,\,a)\}_{l=1,...,n}$ for fixed $n\leq N$, is a martingale
difference sequence for all $(s,\,a)\in\mathbb{K}$. 
\begin{lem}
\label{Lemma 5.2}\cite[Lemma 28]{Even-Dar2004} Given a fixed $n\geq\tau_{m}$,
for any $i\in[\tau_{m},\,n]$, define
\[
\eta_{i}^{m,n}(s,\,a):=\theta{}_{k}^{i+\tau_{m}}(s,\,a)\prod_{j=i+\tau_{m}+1}^{n}[1-\zeta_{t+1}^{j}(s,\,a)].
\]
Let $\tilde{w}_{i+\tau_{m}}^{n}(s,\,a):=\eta_{i}^{m,n}(s,\,a)\,\zeta_{t+1}^{i+\tau_{m}}(s,\,a)$
so that $|Z_{t}^{l,\tau_{m}}(s,\,a)|=\sum_{i=1}^{l}\tilde{w}_{i+\tau_{m}}^{n}(s,\,a)$.
Then, for all $(s,\,a)\in\mathbb{K}$, we have: (i) for any $n\in[\tau_{m+1},\,\tau_{m+2}]$,
the random variable $\tilde{w}_{i+\tau_{m}}^{n}(s,\,a)$ has zero
mean and is bounded by $((1-\varepsilon)^{k}(\tau_{m})^{k})^{-1}V_{\max}$;
(ii) for any $n\in[\tau_{m+1},\,\tau_{m+2}]$ and $1\leq l\leq n$,
$Z_{t}^{l,\tau_{m}}(s,\,a)$ is a martingale difference sequence satisfying
$|Z_{t}^{l,\tau_{m}}(s,\,a)-Z_{t}^{l-1,\tau_{m}}(s,\,a)|\leq((1-\varepsilon)^{k}(\tau_{m})^{k})^{-1}V_{\max}.$
\end{lem}

Based on Lemma \ref{Lemma 5.2} and \cite{Azuma1967}, we obtain a
high probability bound on $|Z_{T}^{n,\tau_{m}}(s,\,a)|$ w.r.t. $m\geq 1$ and $n\in[\tau_{m+1},\,\tau_{m+2}]$ by the Azuma-Hoeffding
inequality, we also derive a selection rule for choosing $\tau_{0}$ and $\{\tau_{m}\}_{m\geq 1}$.
\begin{lem}
\label{Lemma 5.9-1} Given $0<\delta<1$, we have: (i)
\[
\mathbb{P}\left[\forall n\in[\tau_{m+1},\,\tau_{m+2}]:\,\forall(s,\,a)\in\mathbb{K}:\,|Z_{T}^{n,\tau_{m}}(s,\,a)|\leq(1-\frac{2}{e})\beta_{T}D_{m}\right]\geq1-\delta(1-\varepsilon),
\]
for \[\tau_{m}=\Theta\left(\left(\frac{V_{\max}^{2}\ln(V_{\max}|\mathbb{S}||\mathbb{A}|/[\delta\beta_{T}D_{m}(1-\varepsilon)])}{(\beta_{T})^{2}D_{m}(1-\varepsilon)^{1+3k}}\right)^{1/k}\right),\]
and (ii) 
\[
\mathbb{P}\left[\forall m\in[1,\,\frac{1}{1-\varepsilon}],\,\forall n\in[\tau_{m+1},\,\tau_{m+2}],\,\forall(s,\,a)\in\mathbb{K}:\,|Z_{T}^{n,\tau_{m}}(s,\,a)|\leq\tilde{\varepsilon}\right]\geq1-\delta,
\]
for  \[\tau_{0}=\Theta\left(\left(\frac{V_{\max}^{2}\ln(V_{\max}|\mathbb{S}||\mathbb{A}|/[\delta\beta_{T}\tilde{\varepsilon}(1-\varepsilon)])}{(\beta_{T})^{2}\tilde{\varepsilon}^{2}(1-\varepsilon){}^{1+3k}}\right)^{1/k}\right).\]
\end{lem}

\begin{proof}
First, note that $\{Z_{t}^{l,\tau_{m}}(s,\,a)\}_{l=1,...,n}$ is a
martingale difference sequence for all $(s,\,a)\in\mathbb{K}$ by
Lemma \ref{Lemma 5.2}. Next, following the proofs of \cite[Lemma 30]{Even-Dar2004}
and \cite[Lemma 31]{Even-Dar2004}, for each state-action pair, we
apply Lemma \ref{Lemma 5.2} and the Azuma-Hoeffding inequality to
$Z_{T}^{n,\tau_{m}}(s,\,a)$ with $c_{i}=\frac{1}{(1-\varepsilon)^{k}(\tau_{m})^{k}}V_{\max},$
for all $i\in[\tau_{m},\,n]$. Then, for any $n\in\left[\tau_{m+1},\,\tau_{m+2}\right]$,
we have
\begin{align}
\mathbb{P}\left[|Z_{T}^{n,\tau_{m}}(s,\,a)|\geq\tilde{\varepsilon}~|~n\in[\tau_{m+1},\,\tau_{m+2}]\right] & \leq2\exp\left(\frac{-\tilde{\varepsilon}^{2}}{2\sum_{i=\tau_{m}+1,\,i\in N^{s,a}}^{n}c_{i}^{2}}\right)\nonumber \\
 & \leq2\exp\left(-C\frac{\tilde{\varepsilon}^{2}\tau_{m}(1-\varepsilon)^{1+3k}}{V_{\max}^{2}}\right),\label{Azuma}
\end{align}
with a constant $C>0$. Let $\tilde{\delta}_{m}$ denote the right hand side
of the inequality (\ref{Azuma}), which holds for $\tau_{m}=\Theta(\ln(1/\tilde{\delta}_{m})V_{\max}^{2}/(1-\varepsilon)^{1+3k}\tilde{\varepsilon}^{2})$.
The union bound gives
\[
\mathbb{P}\left[\forall n\in[\tau_{m+1},\,\tau_{m+2}]:\,Z_{T}^{n,\tau_{m}}(s,\,a)\leq\tilde{\varepsilon}\right]\leq\sum_{n=\tau_{m}+1}^{\tau_{m}+2}\mathbb{P}\left[Z_{T}^{n,\tau_{m}}(s,\,a)\leq\tilde{\varepsilon}\right],
\]
and so taking $\tilde{\delta}_{m}=\frac{\delta(1-\varepsilon)}{(\tau_{m+2}-\tau_{m+1})|\mathbb{S}||A|}$
assures that with probability at least $1-\delta(1-\varepsilon)$,
we have $|Z_{T}^{n,\tau_{m}}(s,\,a)|\leq\tilde{\varepsilon}$ for
all $(s,\,a)\in \mathbb{K}$ and $n\in[\tau_{m+1},\,\tau_{m+2}]$. As a
result, we have
\[
\tau_{m}=\Theta(\ln(1/\tilde{\delta}_{m})V_{\max}^{2}/(1-\varepsilon)^{1+3k}\tilde{\varepsilon}^{2})=\Theta\left(\left(\frac{V_{\max}^{2}\ln(V_{\max}|\mathbb{S}||\mathbb{A}|/[\delta\beta_{T}D_{m}(1-\varepsilon)])}{(\beta_{T})^{2}D_{m}(1-\varepsilon){}^{1+3k}}\right)^{1/k}\right).
\]
Setting $\tilde{\varepsilon}=(1-2/e)\beta_{T}D_{m}$ gives the desired
bound in Lemma \ref{Lemma 5.9-1}(i). For Lemma \ref{Lemma 5.9-1}(ii),
we know that
\[
\mathbb{P}\left[\forall n\in[\tau_{m+1},\,\tau_{m+2}]:\,|Z_{T}^{n,\tau_{m}}|\geq(1-\frac{2}{e})\beta_{T}D_{m}\right]\leq\frac{\delta}{m},
\]
and obviously
\[
\mathbb{P}\left[\forall n\in[\tau_{m+1},\,\tau_{m+2}]:\,|Z_{T}^{n,\tau_{m}}|\geq D_{m}\right]\leq\frac{\delta}{m}.
\]
Using the union bound again shows that
\[
\mathbb{P}\left[\forall m\leq\frac{1}{1-\varepsilon},\,\forall n\in[\tau_{m+1},\,\tau_{m+2}],\,|Z_{T}^{n,\tau_{m}}|\geq\tilde{\varepsilon}\right]\leq\sum_{m=1}^{M}\mathbb{P}\left[\forall n\in[\tau_{m+1},\,\tau_{m+2}],\,|Z_{T}^{n,\tau_{m}}|\geq\tilde{\varepsilon}\right]\leq\delta,
\]
where $\tilde{\varepsilon}=D_{m}$. We replace $D_{m}$ with $\tilde{\varepsilon}$
in Lemma \ref{Lemma 5.9-1}(i) to obtain \[\tau_{0}=\Theta\left(\left(\frac{V_{\max}^{2}\ln(V_{\max}|\mathbb{S}||\mathbb{A}|/[\delta\beta_{T}\tilde{\varepsilon}(1-\varepsilon)])}{(\beta_{T})^{2}\tilde{\varepsilon}^{2}(1-\varepsilon){}^{1+3k}}\right)^{1/k}.\right)\]
\end{proof}
\textbf{Step 4: Completing the proof by combining Steps 1 through
3.} This step completes the proof of Theorem \ref{Theorem 4.5-1}.
The following Lemma \ref{Fact 5.11-1} is a standard fact about numerical
sequences that is used to derive the final convergence rate in Theorem
\ref{Theorem 4.5-1}.
\begin{lem}
\cite[Lemma 32]{Even-Dar2004}\label{Fact 5.11-1} Let $a_{m+1}=a_{m}+\frac{1}{1-\varepsilon}(a_{m})^{k}=a_{0}+\sum_{i=0}^{m}\frac{1}{1-\varepsilon}(a_{i})^{k}.$
Then, for any $k\in(0,\,1)$, $a_{m}=O\left(((a_{0})^{k}+\frac{1}{1-\varepsilon}m)^{\frac{1}{1-k}}\right)=O\left(a_{0}+(\frac{1}{1-\varepsilon}m)^{\frac{1}{1-k}}\right)$.
\end{lem}

Based on Lemma \ref{Fact 5.11-1}, we set $a_{0}$ to be $\tau_{0}$
in Lemma \ref{Lemma 5.9-1}, and we have
\begin{align*}
 & \mathbb{P}\left[\forall n\in[\tau_{m+1},\,\tau_{m+2}]:\,\forall(s,\,a)\in\mathbb{K}:\,|Z_{T}^{n,\tau_{m}}(s,\,a)|\leq(1-\frac{2}{e})\beta_{T}D_{m}\right]\\
\leq & \mathbb{P}\left[\forall n\in[\tau_{m+1},\,\tau_{m+2}]:\,\forall(s,\,a)\in\mathbb{K}:\,|Z_{T}^{n,\tau_{m}}(s,\,a)|+Y_{T}^{n,\tau_{m}}(s,\,a)\leq D_{m+1}\right]\\
\leq & \mathbb{P}\left[\forall n\in[\tau_{m+1},\,\tau_{m+2}]:\,\|Q_{T}^{n}-Q^{*}\|_{\infty}\leq D_{m+1}\right]\\
\leq & \mathbb{P}\left[\forall n\in[\tau_{m+1},\,\tau_{m+2}]:\,\|Q_{T}^{n}-Q^{*}\|_{\infty}\leq D_{m}\right]\\
\leq & \mathbb{P}\left[\forall n\in[\tau_{m+1},\,\tau_{m+2}]:\,\|Q_{T}^{n}-Q^{*}\|_{2}\leq\sqrt{|\mathbb{S}||\mathbb{A}|}D_{m}\right],
\end{align*}
where the first above inequality holds based on Lemma \ref{Lemma 5.8-1}, the second one holds based on Lemma \ref{Lemma 5.5-2}, and the third one holds based on the definition of sequence $\{D_{m}\}_{m\geq 1}$. Choose $\bar{\varepsilon}$ to satisfy $(1-\frac{2}{e})\beta_{T}D_{m}=\bar{\varepsilon}\leq D_{m}$,
then we have by Lemma \ref{Lemma 5.9-1}(ii) that
\[
\mathbb{P}\left[\forall m\in[1,\,\frac{1}{1-\varepsilon}],\,\forall n\in[\tau_{m+1},\,\tau_{m+2}],\,\forall(s,\,a)\in\mathbb{K}:\,\|Q_{T}^{n}-Q^{*}\|_{2}\leq\sqrt{|\mathbb{S}||\mathbb{A}|}D_{m}\right]\geq1-\delta,
\]
with \[\tau_{0}=\Theta\left(\left(\frac{V_{\max}^{2}\ln(V_{\max}|\mathbb{S}||\mathbb{A}|/[\delta\beta_{T}\bar{\varepsilon}(1-\varepsilon)])}{(\beta_{T})^{2}(\bar{\varepsilon})^{2}(1-\varepsilon){}^{1+3k}}\right)^{1/k}\right).\]
Since this statement holds for all $m\in[1,\,\frac{1}{1-\varepsilon}]$,
based on Lemma \ref{Fact 5.11-1}, we have
\begin{equation}
\mathbb{P}\left[\forall n\in[\tau_{m+1},\,\tau_{m+2}],\,\forall(s,\,a)\in\mathbb{K}:\,\|Q_{T}^{n}-Q^{*}\|_{2}\leq\sqrt{|\mathbb{S}||\mathbb{A}|}D_{m}\right]\geq1-\delta.\label{pb}
\end{equation}
Set $\tilde{\varepsilon}$ such that $\sqrt{|\mathbb{S}||\mathbb{A}|}D_{m}\leq\tilde{\varepsilon}$
and $D_{m}=V_{\max}(1-\beta_{T})^{m}$, we have $m\geq(1/\beta_{T})\ln(V_{\max}\sqrt{|\mathbb{S}||\mathbb{A}|}/\tilde{\varepsilon})$
and so
\[
\tau_{m}=\Theta\left(\left(\frac{V_{\max}^{2}\ln(V_{\max}|\mathbb{S}||\mathbb{A}|/[\delta\beta_{T}\bar{\varepsilon}(1-\varepsilon)])}{(\beta_{T})^{2}(\bar{\varepsilon})^{2}(1-\varepsilon){}^{1+3k}}\right)^{1/k}+\left(\frac{1}{(1-\varepsilon)\beta_{T}}\ln\left(\frac{V_{\max}\sqrt{|\mathbb{S}||\mathbb{A}|}}{\tilde{\varepsilon}}\right)\right)^{\frac{1}{1-k}}\right).
\]
Since the probability bound (\ref{pb}) holds for all $n\in[\tau_{m+1},\,\tau_{m+2}]$,
if we replace $\bar{\varepsilon}$ with $\tilde{\varepsilon}/\sqrt{|\mathbb{S}||\mathbb{A}|}$,
we get the desired result since $\bar{\varepsilon}\leq\tilde{\varepsilon}/\sqrt{|\mathbb{S}||\mathbb{A}|}$.

\section{Numerical experiments}

We illustrate the application of RaQL with an infinite-horizon inventory control problem. In practice, RaQL finds the optimal risk-aware ordering policy $\pi^{\ast}: \mathbb{S}\rightarrow\mathbb{A}$, which is more reliable than the standard one because it is sensitive to low probability events with extremely high random demand.
In each stage, we first observe the current stock $s\in\mathbb{S}$ in inventory, then order $a\in\mathbb{A}$ new units, after which a random demand $D$ is realized. The new inventory level in the next stage is \[s^{\prime} = \max\{0,\,\min\{s + a,\, S_{\max}\}-D\},\] where $S_{\max}$ denotes the largest state in $\mathbb{S}:=\{0,1,...,S_{\max}\}$. The random cost is \[c(s,\,a,\,D) =  \tilde{c}\cdot a + b\cdot\max\{D-s-a,\,0\} - p\cdot\min\{s + a,\,D\},\] where $\tilde{c}$, $p$ and $b$ are the unit order cost, selling price, and backorder cost, respectively. 

For our experiments, we choose $\tilde{c}=3$, $p=5$, $b=4$, $D$ is uniform on $\{1,2,...,10\}$, and the finite state/action spaces are: $\mathbb{S} = \{0,1,...,19\}$ and $\mathbb{A}=\{1,2,...,10\}$. We set the discount factor to be $\gamma = 0.1$, and we assume that all model parameters (costs, price, and transition probabilities) are all stationary. In these experiments, we evaluate the performance in terms of the
relative error $\|Q_{T}^{n}-Q^{*}\|_{2}/\|Q^{*}\|_{2},\,n\leq N$.
Here, we obtain $Q^{*}$ exactly by doing risk-aware DP (as proposed
in \cite{Ruszczynski2010}) where in each iteration the risk-aware
Bellman operator is computed by exactly solving a stochastic saddle-point
optimization problem (see $\mathcal{T}_{G}$ in (\ref{Dynamic programming 1})).
First, we verify the convergence of our algorithm for a few different
risk measures, and then we compare the performance with standard risk-neutral
$Q$-learning. These results confirm the almost sure convergence of
our algorithm as well as its competitive convergence rate. We record and compare the computation time required to reach the same
relative error for RaQL with CVaR and standard $Q$-learning. We also compare the reliability of risk-aware policy and risk-neutral policy by showing how the risk-aware policy reduces the variance of expected cost when the demand is generated from the underlying distribution. Second,
we test the performance of our algorithm against risk-sensitive $Q$-learning
(RsQL) as proposed in \cite{Shen2014} for the entropic risk measure,
since both methods can be applied. This comparison reveals the advantages
of RaQL both in terms of computational efficiency and accuracy. Third,
we compare SASP and stochastic subgradient descent for risk estimation.
This comparison demonstrates that SASP is better suited for estimation
of complex risk measures.

Throughout the experiments, we conduct
$50$ simulation runs for each implementation of $Q$-learning type
algorithms (RaQL, standard $Q$-learning, and RsQL), and record the
mean and standard deviation of relative errors among the simulation
runs. The experiments were performed on a generic laptop with Intel
Core i7 processor, 8GM RAM, on a 64-bit Windows 8 operating system
running Matlab R2015a and CPLEX Studio 12.5. 

\subsection{Experiment I: Risk-aware vs. Risk-neutral }

\subsubsection{Convergence rate comparison}

We intend to show that a variety of risk measures fit into our RaQL
framework, and also to show that RaQL has a convergence rate similar
to risk-neutral $Q$-learning. We consider
CVaR and absolute semi-deviation. We set the number of outer iterations
to be $N=10000$, and the number of inner iterations to be $T=100$.
In these experiments, Risk-aware DP terminates after finding
an $\epsilon$-optimal policy with $\epsilon=0.01$. We use a linear
learning rate i.e. $k=1$, and set $\alpha=0.1$ for CVaR, and $r=0.5$
for absolute semi-deviation. 

As shown in Figure 1, RaQL converges almost surely to the optimal
$Q$-value as expected. Moreover, in this experiment, the convergence
rate of RaQL matches classical $Q$-learning. In Figure 1, the error
bars represent the standard deviation from simulation.

\begin{figure}
\begin{centering}
\includegraphics[scale=0.6]{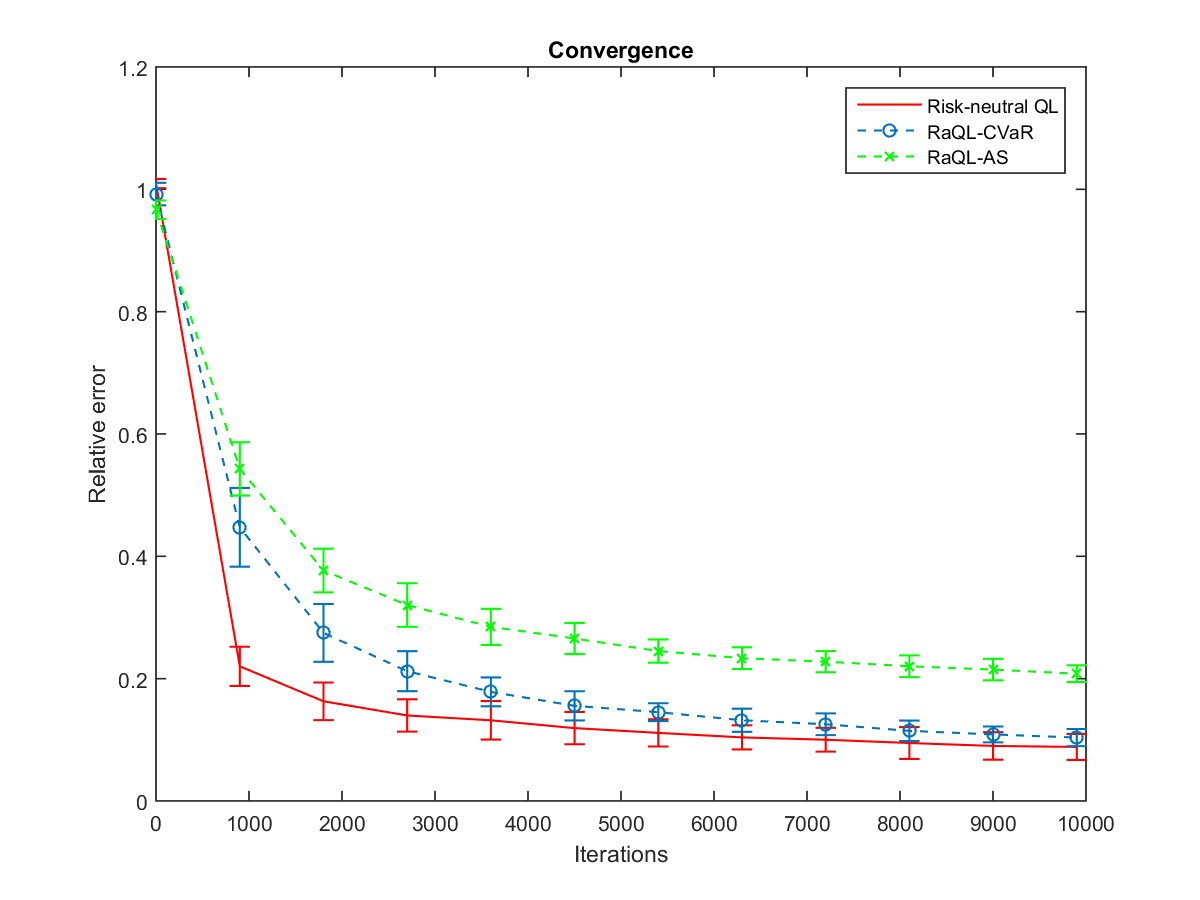}
\par\end{centering}
\caption{Numerical Experiment Result I}
\end{figure}

\subsubsection{Computation time comparison}

In this experiment, we compare the time required for RaQL and standard
$Q$-learning to reach the same precision $\epsilon$ i.e. $\|Q_{T}^{n}-Q^{*}\|_{2}/\|Q^{*}\|_{2}\leq\epsilon$.
Table 2 shows the expected computation time results under $50$ simulations
when choosing different precision levels $\epsilon$, and different
$T$. Table 2 shows that for any $\epsilon$, the expected computation
time for RaQL will decrease with the $T$ selected, and will be close
to that of standard $Q$-learning, which means that RaQL has robust
convergence even when the number of iterations for risk estimation
is small. 

\begin{table}
\begin{centering}
\begin{tabular}{ccccc}
\cline{2-5} \cline{3-5} \cline{4-5} \cline{5-5} 
 & $\epsilon=0.5$ & $\epsilon=0.2$ & $\epsilon=0.15$ & $\epsilon=0.1$\tabularnewline
\hline 
RaQL ($T=100$) & $1.268$s & $3.493$s & $6.798$s & $30.976$s\tabularnewline
\hline 
RaQL ($T=50$) & $0.537$s & $1.061$s & $6.135$s & $5.225$s\tabularnewline
\hline 
RaQL ($T=10$) & $0.119$s & $0.302$s & $0.370$s & $1.329$s\tabularnewline
\hline 
RaQL ($T=5$) & $0.062$s & $0.320$s & $0.286$s & $0.804$s\tabularnewline
\hline 
RaQL ($T=2$) & $0.033$s & $0.169$s & $0.301$s & $0.350$s\tabularnewline
\hline 
RaQL ($T=1$) & $0.022$s & $0.064$s & $0.143$s & $0.529$s\tabularnewline
\hline 
Standard $Q$-learning & $0.027$s & $0.125$s & $0.169$s & $0.374$s\tabularnewline
\hline 
\end{tabular}\\
\par\end{centering}
\caption{Computation Time}
\end{table}

\subsubsection{Policy comparison}
Figure 2 compares the risk-aware ordering policy from RaQL (with CVaR) and the risk-neutral ordering policy from standard $Q$-learning over 500 simulated trajectories. The histograms in Plot 4 show that the risk-aware ordering policy leads to slightly higher expected cost but lower variance. In addition, the right tails of these two distributions reveal that the risk-aware ordering policy is more reliable since it reduces the probability of events with extremely high cost. 
\begin{figure}
	\begin{centering}
		\includegraphics[scale=0.6]{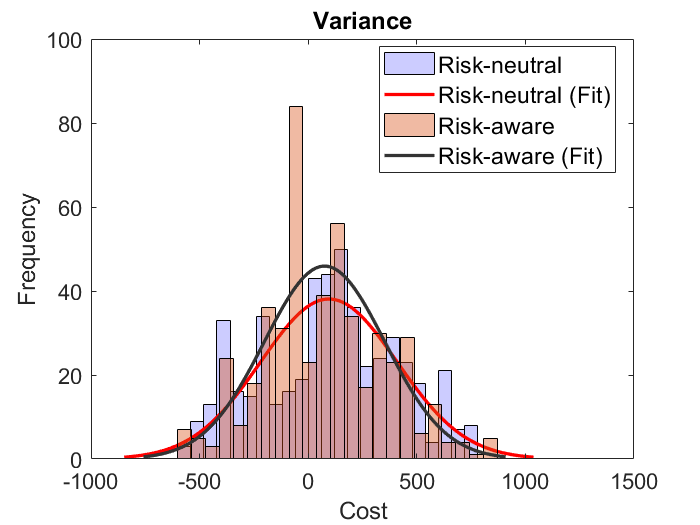}
		\par\end{centering}
	\caption{Policy Comparison}
\end{figure}

\subsection{Experiment II: RaQL vs. RsQL}

In this experiment, we compare the performance of RaQL with risk-sensitive
$Q$-learning (RsQL) as proposed in \cite{Shen2014}. We use the entropic
risk measure (constructed from utility-based shortfall) as in \cite{Foellmer2011,Foellmer2011a}
to compare RaQL and RsQL. An entropic risk measure can be constructed
from the utility function $u(x)=1-\exp(-\lambda\,x),\,\lambda>0$
for $x\in\mathbb{R}$ in OCE from Example \ref{Example 3.3-1}. We
set $\lambda=0.01$, the number of outer iterations to be $N=1\times10^{5}$,
and the number of inner iterations to be $T=10$ for RaQL. The total
number of iterations for RsQL is $1\times10^{5}$. The other settings
remain the same as in Experiment I. Under these settings, RsQL terminates
after 4.559s in expectation and RaQL uses 4.521s in expectation, to
complete the first $1\times10^{4}$ iterations. Figure 2 shows that
RaQL converges faster than RsQL. The convergence rate has also has
lower standard deviation as shown by the error bars. We conjecture
that the inner-outer loop structure of RaQL estimates the risk and
updates the $Q$-values independently, which helps to reduce the bias
in iterative $Q$-learning. In contrast, in RsQL, the risk estimation
and $Q$-value updates are conducted simultaneously which may result
in higher bias. 

\begin{figure}
\begin{centering}
\includegraphics[scale=0.6]{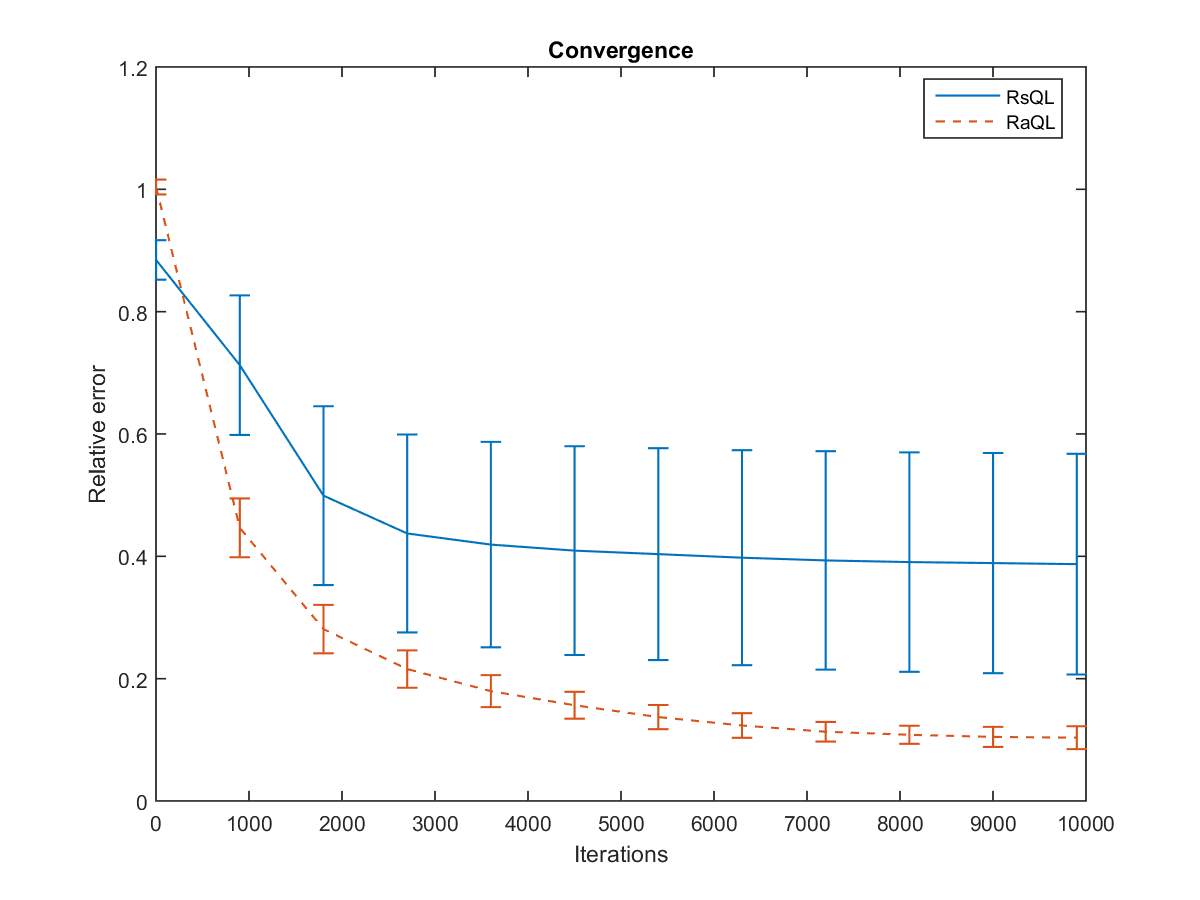}
\par\end{centering}
\caption{Numerical Experiment Result II}
\end{figure}

\subsection{Experiment III: SASP vs. Stochastic subgradient descent}

In this experiment, we compare RaQL with SASP and stochastic subgradient
descent for risk estimation procedure to show that SASP has more accurate risk estimation. In particular, for stochastic
subgradient descent we cut the moving average step (\ref{gnt}), and
change step (\ref{SASP1}) into
\begin{align*}
\left(y_{t+1}^{n}(s_{t}^{n},\,a^{n}),\,z_{t+1}^{n}(s_{t}^{n},\,a^{n})\right)= & \Pi_{\mathcal{Y}\times\mathcal{Z}}\left\{ \left(y_{t}^{n}(s_{t}^{n},\,a^{n}),\,z_{t}^{n}(s_{t}^{n},\,a^{n})\right)\right.\\
 & \left.-\lambda_{t,\alpha}\psi\left(v^{n-1}(s_{t+1}^{n}),\,y_{t}^{n}(s_{t}^{n},\,a^{n}),\,z_{t}^{n}(s_{t}^{n},\,a^{n})\right)\right\} ,
\end{align*}
where the subgradient estimation of the current iteration is combined
with computation of the saddle-point $(y_{t}^{n}(s_{t}^{n},\,a^{n}),\,z_{t}^{n}(s_{t}^{n},\,a^{n})$
(in SASP, the moving average of historical estimations $(y^{n,t}(s_{t}^{n},\,a^{n}),\,z^{n,t}(s_{t}^{n},\,a^{n}))$
is used for this purpose). In this experiment, we set the number of
outer iterations to be $N=3000$ and the number of inner iterations
to be $T=100$, we take a linear learning rate $k=1$, and we set
the step-size for risk estimation to be $\lambda_{t,\alpha}=Ct^{-\alpha}$
with $\alpha=1/2$. We compare the two procedures for a functionally
coherent risk measure (see Example \ref{Example 3.2}). Figure 3 suggests
that RaQL running on SASP has a lower relative error compared to the
modified algorithm which uses stochastic subgradient descent, especially
when the underlying risk measure is non-smooth and degenerate.

\begin{figure}
\begin{centering}
\includegraphics[scale=0.6]{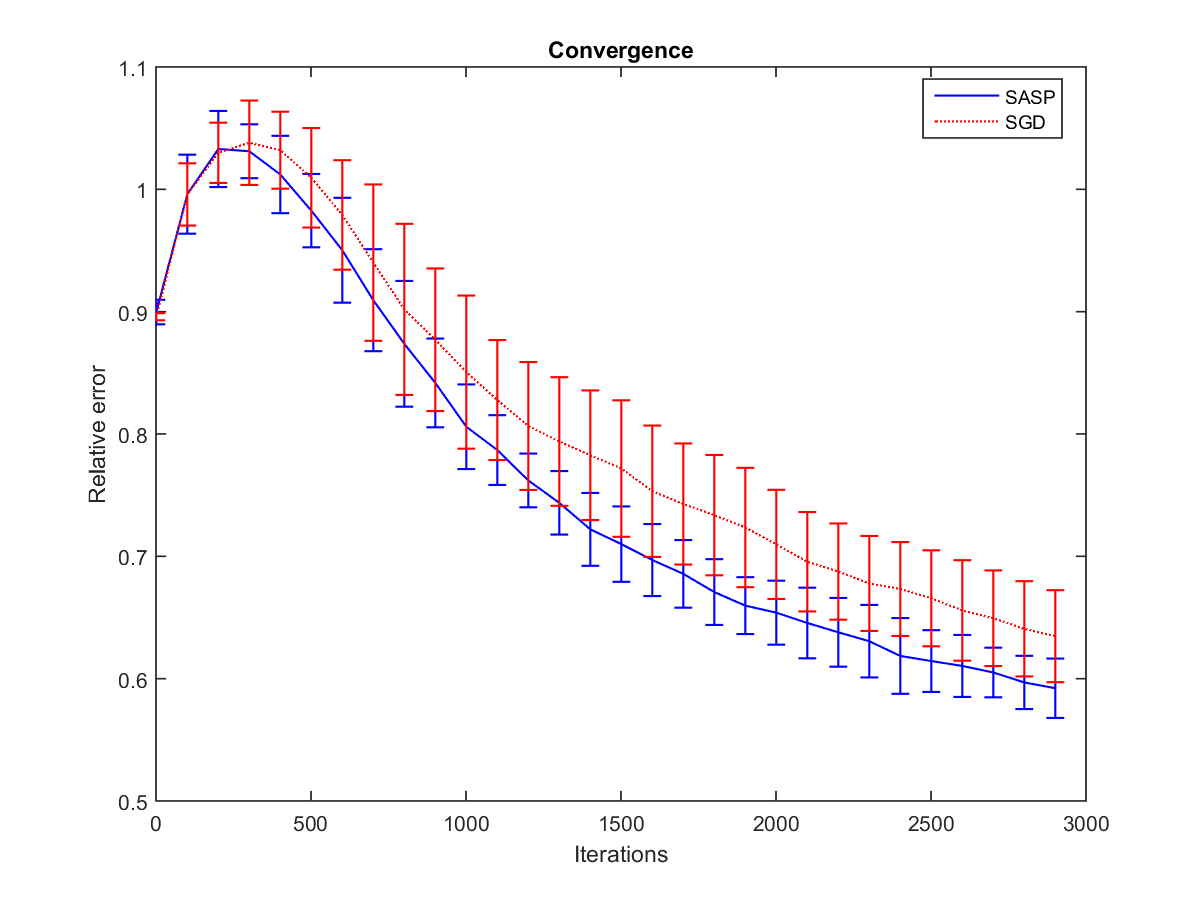}
\par\end{centering}
\caption{Numerical Experiment Result III}
\end{figure}

\section{Conclusion}

We developed a new simulation-based algorithm for finite state/action,
infinite-horizon, risk-aware Markov decision process, called Risk-aware
$Q$-learning (RaQL). This algorithm can be used to solve many real
life risk-aware dynamic optimization problems in areas such as robotics,
sequential online auctions, and infrastructure protection. We demonstrate
that many commonly investigated risk measures (e.g. conditional value-at-risk,
optimized certainty equivalent, absolute semi-deviation, and functionally
coherent risk measures) fit into our framework. We analyze RaQL and
establish both its almost sure convergence as well as its convergence
rate ($\Omega((\ln(1/\delta\epsilon)/\epsilon^{2})^{1/k}+(\ln(1/\epsilon))^{1/(1-k)})$
with probability $1-\delta$, where $\epsilon>0$ is the error tolerance
and $k\in(1/2,\,1]$ is the learning rate). For the case of a linear
learning rate, we get an explicit convergence rate ($\Omega(1/\epsilon)$)
in expectation. Our experiments confirm the almost sure convergence
of RaQL, and also show that RaQL has a convergence rate comparable
to classical $Q$-learning in terms of the required number of outer
loops. Additionally, our experiments illustrate the computational
advantages of RaQL compared with some alternative methods for solving
risk-aware MDPs.

In future research, we will explore new methods for speeding up the
risk estimation subroutine to reduce the overall computational complexity,
and we will also extend RaQL to handle continuous state and action
spaces by incorporating $Q$-value function approximation techniques.

\section*{Appendix}

\noun{Proof of Theorem \ref{Theorem 3.2-1}:} Let $P$ denote the
probability distribution of $X$ and construct $\{h_{z}\}_{z\in\mathcal{Z}}$
satisfying conditions (i)-(iv) in the statement of the theorem. The
stochastic saddle-point formulation in (\ref{Construction G}) is
then
\begin{equation}
\rho(X)=\min_{y\in[\eta_{\min},\,\eta_{\max}]}\max_{z\in\mathcal{Z}}\left\{ y+\text{\ensuremath{\mathbb{E}}}_{P}\left[h_{z}(X,\,y)\right]\right\} .\label{Convex final-1}
\end{equation}
Since $h_{z}$ is $P$-square summable for every $y\in\mathcal{Y}$
and $z\in\text{\ensuremath{\mathcal{Z}}}$, the corresponding function
$G$ (\ref{Construction G}) satisfies Assumption \ref{Assumption 3.3}(i).
Concavity of $h_{z}$ in $z\in\mathcal{Z}$ leads to this function
$G$ satisfying Assumption \ref{Assumption 3.3}(iii). Lipschitz continuity
of $h_{z}$ with modulus $K_{G}-1$ implies that this function $G$
satisfies Assumption \ref{Assumption 3.3}(ii). The condition that
the subgradients of $h_{z}(X-y)$ on $z\in\mathcal{Z}$ and $y\in\mathcal{Y}$
are Borel measurable and uniformly bounded, for any $X\in\mathcal{L}$,
implies that this $G$ satisfies Assumption \ref{Assumption 3.3}(iv). 

Next, we prove that formulation (\ref{Convex final-1}) is a convex
risk measure. Let $\phi(\cdot)$ denote a continuous and subdifferentiable
$\phi$-divergence function for the distance between two probability
distributions $P$ and $P^{\prime}$. We refer to \cite[Table 5]{Postek2016}
for examples of $\phi$-divergence functions. Recall the Fenchel dual
representation of convex risk measures, 
\begin{equation}
\rho(X)=\sup_{P^{\prime}\in\mathbb{P}}\left\{ \mathbb{E}_{P^{\prime}}[X]-\mu(P^{\prime})\right\} ,\label{Convex risk-1}
\end{equation}
where $\mu$ is a convex function satisfying $\inf_{P^{\prime}\in\mathbb{P}}\mu(P^{\prime})=0$,
and $\mathbb{P}$ is the $\phi$-divergence risk envelope,
\[
\mathbb{P}:=\left\{ P^{\prime}:\,P^{\prime}\geq0,\,1^{\top}P^{\prime}=1,\,\int_{\Omega}\phi\left(\frac{dP^{\prime}}{dP}\right)dP\leq\beta\right\} ,
\]
consisting of all probability distributions with $\phi$-divergence
from $P$ bounded by $\beta>0$. Let $\phi^{\ast}$ denote the convex
conjugate of $\phi$ defined as $\phi^{\ast}(X):=\sup_{P^{\prime}\in\mathbb{P}}\left\{ \mathbb{E}_{P^{\prime}}\left[X\right]-\phi(P^{\prime})\right\} $.
Based on the results for $\phi$-divergence risk envelopes constructed
in \cite{Ben-Tal2013,Ben-Tal2015,Postek2016}, any convex risk measure
(\ref{Convex risk-1}) with corresponding set $\mathbb{P}$ can be
reformulated as:
\begin{equation}
\rho(X)=\inf_{b\geq0,\,y\in\mathbb{R}}\left\{ y+b\text{\ensuremath{\beta}}+b\text{\ensuremath{\mathbb{E}}}_{P}\left[\phi^{*}\left(\frac{X-y}{b}\right)\right]\right\} .\label{Counterpart-1}
\end{equation}
Inspired by the minimax risk measure investigated in \cite{Shapiro2002,Shapiro2004,Shapiro2005,Shapiro2012},
we develop an extended variant for (\ref{Counterpart-1}). Let $\phi_{z}$
denote a family of divergence functions parameterized by $z\in\mathcal{Z}$
that is concave in $z\in\mathcal{Z}$ for all fixed $X\in\mathcal{L}$,
and let $\phi_{z}^{\ast}$ denote their corresponding convex conjugates.
Define
\[
\mathbb{P}_{z}:=\left\{ P^{\prime}:\,P^{\prime}\geq0,\,1^{\top}P^{\prime}=1,\,\int_{\Omega}\phi_{z}\left(\frac{dP^{\prime}}{dP}\right)dP\leq\beta\right\} ,
\]
to be the set of probability distributions with bounded divergence
with respect to $\phi_{z},\,z\in\mathcal{Z}$, and set $\mathbb{P}_{\mathcal{Z}}=\bigcup_{z\in\mathcal{Z}}\mathbb{P}_{z}$.
The equivalent form for (\ref{Convex risk-1}) with the set $\mathbb{P}_{\mathcal{Z}}$
is now
\begin{equation}
\rho(X)=\min_{b\geq0,\,y\in\mathbb{R}}\max_{z\in\mathcal{Z}}\left\{ y+b\text{\ensuremath{\beta}}+b\text{\ensuremath{\mathbb{E}}}_{P}\left[\phi_{z}^{*}\left(\frac{X-y}{b}\right)\right]\right\} .\label{Counterpart}
\end{equation}
To complete the proof, given a constructed $\{h_{z}\}_{z\in\mathcal{Z}}$,
and if we choose the $\phi$-divergence function with its convex conjugate
$\phi^{\ast}$ satisfying
\[
\phi_{z}^{*}\left(\frac{X-y}{b}\right)=\frac{h_{z}(X,\,y)}{b}-\beta,
\]
for any $y\in\mathbb{R}$ and $b\geq0$, then the formulation (\ref{Counterpart})
is equivalent to formulation (\ref{Convex final-1}). Thus formulation
(\ref{Convex final-1}) is a convex risk measure. \\
\noun{}\\
\noun{Proof of Theorem \ref{Theorem 4.6}}: In this part, we detail
the procedures to prove Theorem \ref{Theorem 4.6}. As a remark, the natural logarithm term $e$ in (\ref{Difference})
of Lemma \ref{Lemma 5.8-1} is specific to a polynomial learning rate,
while for a linear learning rate we have a new relationship between
$\tau_{m}$ and $\tau_{m+1}$. Thus, we must construct a different
bound on $\{Y_{T}^{n,\tau_{m}}\}_{\tau_{m}\leq n\leq N}$, which is
defined in Eq. (\ref{Process Y-1}). We first derive
the convergence rate of process $Y_{T}^{n,\tau_{m}}(s,\,a)$ w.r.t.
$T$. We prove the result by applying the same argument as in the
proof of Lemma \ref{Lemma 5.8-1}, and combining the arguments of
\cite[Lemma 22]{Even-Dar2004}.
\begin{lem}
Given any $m\geq1$, assume that for all $n\geq\tau_{m}$ we have
$Y_{T}^{n,\tau_{m}}\leq D_{m}$. Then for any $n\geq(2+\Psi)\tau_{m}=\tau_{m+1}$,
we have, 
\[
Y_{T}^{n,\tau_{m}}(s,\,a)\leq D_{m}\left(K_{G}\left\{ 1-\gamma-\sqrt{\frac{C(\tau_{\ast}(T))^{-\alpha}}{\kappa(1-C(\tau_{\ast}(T))^{-\alpha}K_{\psi}^{(1)}\kappa)}}-K_{G}K_{S}\right\} +\frac{2}{2+\Psi}\beta_{T}\right),
\]
for all $(s,\,a)\in\mathbb{K}$, where $\beta_{T}$ is defined in
(\ref{Beta-1}) with $T$ satisfying conditions (\ref{C1}) and (\ref{C2}),
$\Psi$ is any positive constant and $K_{s}$ is defined in (\ref{Relationship}). 
\end{lem}

The following Lemma enables the use of Azuma-Hoeffding inequality.
Lemma \ref{Lemma 8.2 } can be prove by applying the same argument
in Lemma \ref{Lemma 5.2}, where we set $k=1$ because of linear learning
rate.
\begin{lem}
\label{Lemma 8.2 } For any $n\geq\tau_{m}$ and $1\leq l\leq n$
we have that $\{Z_{t}^{l,\tau_{m}}(s,\,a)\}_{l=1,...,n}$ is a martingale
difference sequence, which satisfies $|Z_{t}^{l,\tau_{m}}(s,\,a)-Z_{t}^{l-1,\tau_{m}}(s,\,a)|\leq\frac{V_{\max}}{(1-\varepsilon)\tau_{m}},$
for any $t\leq T$.
\end{lem}

\begin{lem}
\label{Lemma 5.14} Given $0<\delta<1$, we have (i)
\begin{equation}
\mathbb{P}\left[\forall n\in[\tau_{m+1},\,\tau_{m+2}]:\,Z_{T}^{n,\tau}(s,\,a)\leq\frac{\Psi}{2+\Psi}\beta_{T}D_{m}\right]\geq1-\delta(1-\varepsilon),\label{statement}
\end{equation}
where \textup{$\tau_{m}=\Theta\left(\frac{V_{\max}^{2}\ln(V_{\max}|\mathbb{S}||\mathbb{A}|/[\Psi\delta\beta_{T}D_{m}(1-\varepsilon)])}{\Psi^{2}\beta_{T}D_{m}^{2}(1-\varepsilon)^{2}}\right)$},
and (ii)
\[
\mathbb{P}\left[\forall m\in[1,\,\frac{1}{1-\varepsilon}],\,\forall n\in[\tau_{m+1},\,\tau_{m+2}]:\,|Z_{T}^{n,\tau}|\leq\tilde{\varepsilon}\right]\geq1-\delta,
\]
where $\tau_{0}=\Theta\left(\frac{V_{\max}^{2}\ln(V_{\max}|\mathbb{S}||\mathbb{A}|/[\Psi\delta\beta_{T}\tilde{\varepsilon}(1-\varepsilon)])}{\Psi^{2}\beta_{T}\tilde{\varepsilon}^{2}}\right)$. 
\end{lem}

\begin{proof}
Following the proofs of \cite[Lemma 37]{Even-Dar2004} and \cite[Lemma 38]{Even-Dar2004},
we set$c_{i}=\Theta\left(\frac{V_{\max}}{(1-\text{\ensuremath{\varepsilon})}\tau_{m}}\right)$,
for any $n\geq\tau_{m+1}$, therefore we obtain
\begin{align*}
\mathbb{P}\left[|Z_{T}^{n,\tau}|\geq\tilde{\varepsilon}\right] & \leq2\exp\left(\frac{-\tilde{\varepsilon}^{2}}{2\sum_{i=\tau_{m}+1,\,i\in N^{s,a}}^{n}c_{i}^{2}}\right)\leq2\exp\left(-c\frac{\tilde{\varepsilon}^{2}\tau_{m}(1+\Psi)}{V_{\max}^{2}}\right),
\end{align*}
for some constant $c>0$. Let us define the following variable
\[
\Xi^{n}(s,\,a)=\begin{cases}
1, & \text{\ensuremath{\theta}}_{t,k}^{n}\neq0\\
0 & \textrm{otherwise},
\end{cases}
\]
where $k$ is fixed. Using the union bound and the fact in an interval
of length $\frac{1+\Psi}{1-\varepsilon}\tau_{m}$, each state-action
pair is visited at least $(1+\Psi)\tau_{m}$ times with certainty
according to \cite[Lemma 37]{Even-Dar2004}, we get
\begin{align*}
\mathbb{P}\left[\forall n\in[\tau_{m+1},\,\tau_{m+2}]:\,|Z_{T}^{n,\tau_{m}}|\geq\tilde{\varepsilon}\right] & \leq\mathbb{P}\left[\forall n\geq(\frac{1+\Psi}{1-\varepsilon}+1)\tau_{m}:\,|Z_{T}^{n,\tau_{m}}|\geq\tilde{\varepsilon}\right]\\
 & \leq\sum_{n=((1+\Psi)/(1-\varepsilon)+1)\tau_{m}}^{\infty}\mathbb{P}\left[|Z_{T}^{n,\tau_{m}}(s,\,a)|\geq\tilde{\varepsilon}\right]\\
 & \leq2\sum_{n=((1+\Psi)/(1-\varepsilon)+1)\tau_{m}}^{\infty}\Xi^{n}(s,\,a)\exp\left(-c\frac{\tau_{m}(1+\Psi)\tilde{\varepsilon}^{2}}{V_{\max}^{2}}\right)\\
 & \leq2\exp\left(-c\frac{((1+\Psi)\tau_{m})\tilde{\varepsilon}^{2}}{V_{\max}^{2}}\right)\sum_{n=0}^{\infty}\exp\left(-\frac{n\tilde{\varepsilon}^{2}}{2V_{\max}^{2}}\right)\\
 & =\frac{2\exp\left(-c\frac{(1+\Psi)\tau_{m}\tilde{\varepsilon}^{2}}{V_{\max}^{2}}\right)}{1-\exp\left(-\frac{\tilde{\varepsilon}^{2}}{V_{\max}^{2}}\right)}\\
 & =\Theta\left(\frac{\exp\left(-\frac{c^{\prime}\tau_{m}\tilde{\varepsilon}^{2}}{V_{\max}^{2}}\right)}{\tilde{\varepsilon}^{2}}V_{\max}^{2}\right),
\end{align*}
for some positive constant $c^{\prime}$. Controlling $\delta$ by
setting
\[
\frac{\delta(1-\varepsilon)}{|\mathbb{S}||\mathbb{A}|}=\Theta\left(\frac{\exp\left(-\frac{c^{\prime}\tau_{m}\tilde{\varepsilon}^{2}}{V_{\max}^{2}}\right)}{\tilde{\varepsilon}^{2}}V_{\max}^{2}\right),
\]
which holds for $\tau_{m}=\Theta\left(\frac{V_{\max}^{2}\ln(V_{\max}|\mathbb{S}||\mathbb{A}|/(\delta D_{m}(1-\varepsilon))}{\beta_{T}D_{m}}\right)$,
and $\tilde{\varepsilon}=\frac{\Psi}{2+\Psi}\beta_{T}D_{m}$ assures
us that for every $t\geq\tau_{m+1}$ with probability at least $1-\delta(1-\varepsilon)$
the statement (\ref{statement}) holds at every state-action pair.
\end{proof}
Theorem \ref{Theorem 4.6} follows from Lemma \ref{Lemma 5.14}, and
the algebraic identity in the proof of \cite[Theorem 5]{Even-Dar2004}
that
\[
a_{k+1}=a_{k}+\frac{(1+\Psi)}{1-\varepsilon}a_{k}=a_{0}(\frac{(1+\Psi)}{1-\varepsilon}+1)^{k}.
\]
The detailed proof follows the same procedures as the proof of Theorem
\ref{Theorem 4.5-1}. \\
\noun{}\\
\noun{Proof of Theorem \ref{Theorem 4.7}}: To start, we investigate
the convergence rate of risk estimation step w.r.t. $t$ by stochastic
approximation. We first refer to the convergence rate analysis of
Algorithm 2 in \cite{Nemirovski2005}. As a measure of the quality
of a candidate solution $(y,\,z)\in\mathcal{Y}\times\mathcal{Z}$,
we use the duality gap $d(y,\,z)$ proposed by \cite[Section 2.1.2]{Nemirovski2005}.
Let $\overline{\Phi}(y)=\max_{z\in\mathcal{Z}}\mathbb{E}_{P}\left[G(X,\,y,\,z)\right]$,
and $\underline{\Phi}(z)=\min_{y\in\mathcal{Y}}\mathbb{E}_{P}\left[G(X,\,y,\,z)\right]$,
for any fixed $X\in\mathcal{L}$, and
\begin{align*}
d(y,\,z) & :=[\overline{\Phi}(y)-\min_{y\in\mathcal{Y}}\overline{\Phi}(y)]+[\max_{z\in\mathcal{Z}}\underline{\Phi}(z)-\underline{\Phi}(z)]=\overline{\Phi}(y)-\underline{\Phi}(z).
\end{align*}
The next theorem gives the convergence rate of SASP.
\begin{thm}
\label{Theorem 3.2}\cite[Theorem 1]{Nemirovski2005} Suppose Assumption
\ref{Assumption 3.3} holds, set the step-size for all $t$ as $\lambda_{t,\alpha}=Ct^{-\alpha},\,\alpha\in(0,\,1]$,
then for every $t>1$, we have
\begin{align}
d(y^{t},\,z^{t})\leq & \left[K_{\mathcal{Y}}H_{\mathcal{Y}}^{-1}+K_{\mathcal{Z}}H_{\mathcal{Z}}^{-1}\right]\frac{t^{\alpha}}{C\left(t-\tau_{\ast}(t)+1\right)}+\frac{(K_{\mathcal{Y}}+K_{\mathcal{Z}})L}{\sqrt{t-\tau_{\ast}(t)+1}}\nonumber \\
 & +C(K_{\mathcal{Y}}+K_{\mathcal{Z}})^{2}L^{2}\left[H_{\mathcal{Y}}K_{\mathcal{Y}}+H_{\mathcal{Z}}K_{\mathcal{Z}}\right]\tau_{\ast}^{-\alpha}(t).\label{Offline bound-1}
\end{align}
\end{thm}

\begin{lem}
\label{Lemma 4.2} Suppose Assumption \ref{Assumption 4.5} holds,
for all $(s,\,a)\in\mathbb{K},\,v\in\mathcal{V}$ and for every $t>1$
and $n\leq N$, we have the upper bound
\begin{align}
 & \mathbb{E}_{s^{\prime}\sim P(\cdot|s,a)}\left[|G\left(v(s^{\prime}),\,y^{n,t}(s,\,a),\,z^{n,t}(s,\,a)\right)-G\left(v(s^{\prime}),\,y^{n,\ast}(s,\,a),\,z^{n,\ast}(s,\,a)\right)|\right]\nonumber \\
\leq & \left[K_{\mathcal{Y}}H_{\mathcal{Y}}^{-1}+K_{\mathcal{Z}}H_{\mathcal{Z}}^{-1}\right]\frac{t^{\alpha}}{C\left(t-\tau_{\ast}(t)+1\right)}+\frac{(K_{\mathcal{Y}}+K_{\mathcal{Z}})L}{\sqrt{t-\tau_{\ast}(t)+1}}\nonumber \\
 & +C(K_{\mathcal{Y}}+K_{\mathcal{Z}})^{2}L^{2}\left[H_{\mathcal{Y}}K_{\mathcal{Y}}+H_{\mathcal{Z}}K_{\mathcal{Z}}\right]\tau_{\ast}^{-\alpha}(t).\label{Right hand}
\end{align}
\end{lem}

\begin{proof}
By the triangle inequality, we know that for all $(s,\,a)\in\mathbb{K},\,v\in\mathcal{V}$
and for every $t>1$ and $n\leq N$,
\begin{align}
 & \mathbb{E}_{s^{\prime}\sim P(\cdot|s,a)}\left[|G\left(v(s^{\prime}),\,y^{n,t}(s,\,a),\,z^{n,t}(s,\,a)\right)-G\left(v(s^{\prime}),\,y^{n,\ast}(s,\,a),\,z^{n,\ast}(s,\,a)\right)|\right]\nonumber \\
\leq & \max_{z\in\mathcal{Z}}\mathbb{E}_{s^{\prime}\sim P(\cdot|s,a)}\left[G\left(v(s^{\prime}),\,y^{n,t}(s,\,a),\,z^{n,t}(s,\,a)\right)\right]\nonumber \\
 & -\min_{y\in\mathcal{Y}}\mathbb{E}_{s^{\prime}\sim P(\cdot|s,a)}\left[G\left(v(s^{\prime}),\,y^{n,t}(s,\,a),\,z^{n,t}(s,\,a)\right)\right].\label{13}
\end{align}
From Theorem \ref{Theorem 3.2}, we know that
\begin{align}
 & \max_{z\in\mathcal{Z}}\mathbb{E}_{s^{\prime}\sim P(\cdot|s,a)}\left[G\left(v(s^{\prime}),\,y^{n,t}(s,\,a),\,z^{n,t}(s,\,a)\right)\right] -\min_{y\in\mathcal{Y}}\mathbb{E}_{s^{\prime}\sim P(\cdot|s,a)}\left[G\left(v(s^{\prime}),\,y^{n,t}(s,\,a),\,z^{n,t}(s,\,a)\right)\right]\nonumber \\
\leq & \left[K_{\mathcal{Y}}H_{\mathcal{Y}}^{-1}+K_{\mathcal{Z}}H_{\mathcal{Z}}^{-1}\right]\frac{t^{\alpha}}{C\left(t-\tau_{\ast}(t)+1\right)}+\frac{(K_{\mathcal{Y}}+K_{\mathcal{Z}})L}{\sqrt{t-\tau_{\ast}(t)+1}}\nonumber \\
 & +C(K_{\mathcal{Y}}+K_{\mathcal{Z}})^{2}L^{2}\left[H_{\mathcal{Y}}K_{\mathcal{Y}}+H_{\mathcal{Z}}K_{\mathcal{Z}}\right]\tau_{\ast}^{-\alpha}(t).\label{Offline bound}
\end{align}
Thus we obtain the desired result. 
\end{proof}
The next lemma bounds $\mathbb{E}\left[\|Q_{t}^{n}-Q^{\ast}\|_{2}^{2}|\mathcal{G}_{t+1}^{n-1}\right]$
w.r.t. $n$, for any $t\leq T$. For simplicity, we use function $f(t)$
to denote the right hand side of (\ref{Right hand}).
\begin{lem}
\label{Lemma 5.16}The sequence $Q_{t}^{n}$, generated by Algorithm
1 satisfies, for any $\kappa_{0},\,\kappa>0$,
\begin{align*}
\mathbb{E}\left[\|Q_{t}^{n}-Q^{\ast}\|_{2}^{2}|\mathcal{G}_{t+1}^{n-1}\right] & \leq\left[1-\frac{1}{n^{k}}(2\varepsilon-\gamma K_{G}\kappa_{0}-\gamma K_{G}\kappa-\gamma K_{G}/\kappa_{0}-K_{S}\,K_{G}^{2}/\kappa)\right]\mathbb{E}\left[\|Q_{T}^{n-1}-Q^{\ast}\|_{2}^{2}|\mathcal{G}_{t+1}^{n-1}\right]\\
 & +\frac{(C_{G}+(\gamma\,f(t))^{2})}{n^{2k}},
\end{align*}
where $C_{G}$ bounds the term
\[
\left\{ 1+\frac{C}{\kappa(1-C\,K_{\psi}^{(1)}\kappa)}+\left[K_{G}(\gamma+K_{S}K_{G})\right]^{2}\right\} \|Q_{T}^{n-1}-Q^{\ast}\|_{2}^{2}\leq C_{G}.
\]
\end{lem}

\begin{proof}
Let us recall that the update of Step 3 in Algorithm 1 is equivalent
to
\[
Q_{t}^{n}=Q_{T}^{n-1}-\text{\ensuremath{\theta}}_{k}^{n}\left[Q_{T}^{n-1}-Q^{\ast}+\xi_{t}^{n}+\epsilon_{t}^{n}+Q^{\ast}-H\left(v^{\ast},\,y^{\ast},\,z^{\ast})\right)\right].
\]
Expanding, we have
\begin{align}
\|Q_{t}^{n}-Q^{\ast}\|_{2}^{2}\leq & \|Q_{T}^{n-1}-Q^{\ast}\|_{2}^{2}+\|\text{\ensuremath{\theta}}_{k}^{n}\left[Q_{T}^{n-1}-Q^{\ast}+\epsilon_{t}^{n}+\xi_{t}^{n}\right]\|_{2}^{2}\nonumber \\
 & -2\left(Q_{T}^{n-1}-Q^{\ast}\right)^{\top}\text{\ensuremath{\theta}}_{k}^{n}\left[Q_{T}^{n-1}-Q^{\ast}+\epsilon_{t}^{n}+\xi_{t}^{n}\right].\label{A.2}
\end{align}
We focus on the cross term $-2\left(Q_{T}^{n-1}-Q^{\ast}\right)^{\top}\text{\ensuremath{\theta}}_{k}^{n}\left[Q_{T}^{n-1}-Q^{\ast}+\epsilon_{t}^{n}+\xi_{t}^{n}\right]$.
First, by the $\varepsilon$-greedy exploration policy, notice that
\begin{equation}
\mathbb{E}_{p}\left[\left(Q_{T}^{n-1}-Q^{\ast}\right)^{\top}\text{\ensuremath{\theta}}_{k}^{n}\left(Q_{T}^{n-1}-Q^{\ast}\right)|\mathcal{G}_{t+1}^{n-1}\right]\geq\frac{\varepsilon}{n^{k}}\|Q_{T}^{n-1}-Q^{\ast}\|_{2}^{2},\label{A.3}
\end{equation}
and by the definition of $\epsilon_{t}^{n}$, we have by Lemma \ref{Lemma 4.2},
that
\begin{align*}
\mathbb{E}\left[\epsilon_{t}^{n}|\mathcal{G}_{t+1}^{n-1}\right] & =\mathbb{E}\left[H\left(v^{n-1},\,y^{n,\ast},\,z^{n,\ast}\right)-H\left(v^{n-1},\,y^{n,t},\,z^{n,t}\right)|\mathcal{G}_{t+1}^{n-1}\right]\\
 & \leq\gamma\,\mathbb{E}_{s^{\prime}\sim P(\cdot|s,a)}\left[|G\left(v^{n-1}(s^{\prime}),\,y^{n,t}(s,\,a),\,z^{n,t}(s,\,a)\right)-G\left(v^{n-1}(s^{\prime}),\,y^{n,\ast}(s,\,a),\,z^{n,\ast}(s,\,a)\right)|\right]\\
 & \leq\gamma\,f(t),
\end{align*}
and
\begin{equation}
\mathbb{E}\left[-\left(Q_{T}^{n-1}-Q^{\ast}\right)^{\top}\text{\ensuremath{\theta}}_{k}^{n}\epsilon_{t}^{n}|\mathcal{G}_{t+1}^{n-1}\right]\leq\frac{\gamma\,f(t)}{n^{k}}\|Q_{T}^{n-1}-Q^{\ast}\|_{2}.\label{A.4}
\end{equation}
Applying the algebraic identity $2ab\leq a^{2}\kappa+b^{2}/\kappa$
for all $\kappa>0$, we see that
\[
\mathbb{E}\left[-2\left(Q_{T}^{n-1}-Q^{\ast}\right)^{\top}\text{\ensuremath{\theta}}_{k}^{n}\epsilon_{t}^{n}|\mathcal{G}_{t+1}^{n-1}\right]\leq\kappa\|Q_{T}^{n-1}-Q^{\ast}\|_{2}^{2}+\frac{(\gamma\,f(t))^{2}}{n^{2k}\kappa}.
\]
We can see that
\begin{align*}
\mathbb{E}\left[-\left(Q_{T}^{n-1}-Q^{\ast}\right)^{\top}\text{\ensuremath{\theta}}_{k}^{n}\xi_{t}^{n}|\mathcal{G}_{t+1}^{n-1}\right]\leq & \frac{\gamma^{2}}{n^{k}}K_{G}\biggl[\|Q_{T}^{n-1}-Q^{\ast}\|_{2}\|Q_{T}^{n-1}-Q^{\ast}\|_{2}\\
 & +\gamma\|Q_{T}^{n-1}-Q^{\ast}\|_{2}\left(\|(y^{n,\ast},\,z^{n,\ast})-(y^{\ast},\,z^{\ast})\|_{2}\right)\biggr].
\end{align*}
Again applying the algebraic identity $2ab\leq a^{2}\kappa+b^{2}/\kappa$
for all $\kappa>0$, we see that for any constants $\kappa_{0},\,\kappa>0$,
\begin{alignat}{1}
\mathbb{E}\left[-2\left(Q_{T}^{n-1}-Q^{\ast}\right)^{\top}\text{\ensuremath{\theta}}_{k}^{n}\xi_{t}^{n}|\mathcal{G}_{t+1}^{n-1}\right]\leq & \frac{\gamma}{n^{k}}K_{G}\left[\left(\kappa_{0}+\kappa\right)\|Q_{T}^{n-1}-Q^{\ast}\|_{2}^{2}\right]\nonumber \\
 & +\frac{\gamma K_{G}}{n^{k}\kappa_{0}}\|Q_{T}^{n-1}-Q^{\ast}\|_{2}^{2}+\frac{K_{G}}{n^{k}\kappa}\left(\|(y^{n,\ast},\,z^{n,\ast})-(y^{\ast},\,z^{\ast})\|_{2}^{2}\right)\nonumber \\
\leq & \frac{\gamma}{n^{k}}K_{G}\left[\left(\kappa_{0}+\kappa\right)\|Q_{T}^{n-1}-Q^{\ast}\|_{2}^{2}\right]\nonumber \\
 & +\frac{\gamma K_{G}}{n^{k}\kappa_{0}}\|Q_{T}^{n-1}-Q^{\ast}\|_{2}^{2}+\frac{K_{S}\,K_{G}^{2}}{n^{k}\kappa}\|Q_{T}^{n-1}-Q^{\ast}\|_{2}^{2}.\label{A5}
\end{alignat}
Finally, there exists $\kappa>0$ such that
\begin{align}
 & \|\text{\ensuremath{\theta}}_{k}^{n}\left[Q_{T}^{n-1}-Q^{\ast}+\epsilon_{t}^{n}+\xi_{t}^{n}\right]\|_{2}^{2}\nonumber \\
\leq & \frac{1}{n^{2k}}\left\{ 1+\frac{C}{\kappa(1-C\,K_{\psi}^{(1)}\kappa)}+\left[K_{G}(\gamma+K_{S}K_{G})\right]^{2}\right\} \|Q_{T}^{n-1}-Q^{\ast}\|_{2}^{2},\label{C_G}
\end{align}
where the above inequality holds based on (\ref{Boundedness}) and
(\ref{Induction}). For simplicity, let $C_{G}$ bound the term
\[
\left\{ 1+\frac{C}{\kappa(1-C\,K_{\psi}^{(1)}\kappa)}+\left[K_{G}(\gamma+K_{S}K_{G})\right]^{2}\right\} \|Q_{T}^{n-1}-Q^{\ast}\|_{2}^{2}.
\]
The statement of the lemma follows by taking expectations of inequalities
(\ref{A.2}), (\ref{A.3}), (\ref{A.4}) and (\ref{A5}), and using
the inequality (\ref{C_G}), and combining.
\end{proof}
The convergence rate in expectation with a linear learning rate is
based on the following result from \cite{Chung1954}, which is useful
for analyzing a specific type of sequence that often arises in recursive
algorithms.
\begin{lem}
\label{Lemma 5.17}\cite{Chung1954} Consider a sequence $\{a^{n}\}$.
Suppose for some $b>1$ and every $n\geq1$ that $a^{n}\leq\left(1-\frac{b}{n}\right)a^{n-1}+\frac{c}{n^{2}}.$Then,
if $d\geq\max\left\{ \frac{c}{b-1},\,a^{0}\right\} $, it follows
that $a^{n}\leq\frac{d}{n}$ for $n\geq1$.
\end{lem}

Based on the results of Lemma \ref{Lemma 5.16}, for linear learning
rate $k=1$, we let $a^{n}=\mathbb{E}\left[\|Q_{t}^{n}-Q^{\ast}\|_{2}^{2}|\mathcal{G}_{t+1}^{n}\right]$
for any $t\leq T$, and choose 
\[
d\geq\max\left\{ \frac{(C_{G}+(\gamma\,f(t))^{2})}{(2\varepsilon-\gamma K_{G}\kappa_{0}-\gamma K_{G}\kappa-\gamma K_{G}/\kappa_{0}-K_{S}\,K_{G}^{2}/\kappa)-1},\,C_{\textrm{max}}^{2}|\mathbb{S}||\mathbb{A}|\right\} ,
\]
in Lemma \ref{Lemma 5.17}, where we also have $a^{0}=\mathbb{E}\left[\|Q_{t}^{1}-Q^{\ast}\|_{2}^{2}|\mathcal{G}_{t+1}^{0}\right]\leq C_{\textrm{max}}^{2}|\mathbb{S}||\mathbb{A}|$.
By setting $\kappa=\kappa_{0}=\varepsilon$, given a small positive
constant $\tilde{\varepsilon}>0$, we have the sample complexity 
\[
N=\Omega\left(\frac{\max\left\{ (C_{G}+(\gamma\,f(t))^{2})\varepsilon/\left((2-2\gamma K_{G})\varepsilon^{2}-K_{G}(\gamma-K_{S}K_{G})-\varepsilon\right),\,C_{\max}^{2}|\mathbb{S}||\mathbb{A}|\right\} }{\tilde{\varepsilon}}\right),
\]
such that $\mathbb{E}\left[\|Q_{t}^{N}-Q^{\ast}\|_{2}^{2}|\mathcal{G}_{t+1}^{N-1}\right]\leq\tilde{\varepsilon}$,
for any $t\leq T$. We have thus proved the desired result summarized
in Theorem \ref{Theorem 4.7}.

\bibliographystyle{plain}
\bibliography{Reference_Q-learning}

\end{document}